\documentclass[journal]{IEEEtran}  
\usepackage{amsmath}
\usepackage{amssymb}
\usepackage{amsthm}
\usepackage{dsfont}

\usepackage{graphics} 
\usepackage{epsfig} 
\usepackage{epstopdf}
\usepackage{mathptmx} 
\usepackage{times} 
\usepackage{amsmath} 
\usepackage{amssymb}  
\usepackage{subfigure}
\usepackage[utf8]{inputenc}
\usepackage[english]{babel}
\usepackage{color}

\newtheorem{thm}{Theorem}
\newtheorem{defn}[thm]{Definition}
\newtheorem{prop}[thm]{Proposition}
\newtheorem{lem}[thm]{Lemma}

\newtheorem{cor}[thm]{Corollary}

\newtheorem{ex}[thm]{Example}
\newtheorem{cex}[thm]{Counterexample}

\newcommand{\mcl}[1]{\mathcal{#1}}
\newcommand{\mbb}[1]{\mathbb{#1}}
\newcommand{\mbf}[1]{\mathbf{#1}}

\newcommand{\R}{\mathbb{R}}

\newcommand{\N}{\mathbb{N}}
\newcommand{\bmat}[1]{\begin{bmatrix}#1\end{bmatrix}}

\newcommand{\eps}{\varepsilon}

\title{\LARGE \bf
Extensions of the Dynamic Programming Framework:\\ Battery Scheduling, Demand Charges, and Renewable Integration
}

\author{Morgan Jones%
	\thanks{M. Jones is with the School for the Engineering of Matter, Transport and Energy, Arizona State University, Tempe, AZ, 85298 USA. e-mail: {\tt \small morgan.c.jones@asu.edu } },
	Matthew M. Peet
	\thanks{M. Peet is with the School for the Engineering of Matter, Transport and Energy, Arizona State University, Tempe, AZ, 85298 USA. e-mail: {\tt \small mpeet@asu.edu } }
}

\begin{document}

\maketitle
\thispagestyle{plain}
\pagestyle{plain}

\begin{abstract}
We consider a general class of Dynamic Programming (DP) problems with non-separable objective functions. We show that for any problem in this class, there exists an augmented-state DP problem which satisfies the Principle of Optimality and the solutions to which yield solutions to the original problem. {Furthermore, we identify a subclass of DP problems with Naturally Forward Separable (NFS) objective functions for which this state-augmentation scheme is tractable.} We extend this framework to stochastic DP problems, {proposing a suitable definition of the Principle of Optimality}. We then apply the resulting algorithms to the problem of optimal battery scheduling with demand charges using a data-based stochastic model for electricity usage and solar generation by the consumer.

%
\end{abstract}

\section{Introduction}

The optimal use of battery storage can be formulated as a discrete time process combined with decision variables and an objective function - a formulation commonly known as Dynamic Programming (DP)~\cite{Bellman}. DP is a class of algorithms that break down complex optimization problems into simpler sequential subproblems, each of which is solved using Bellman's Equation. For DP to work, however, we require that the optimization problem satisfies the \textit{Principle of Optimality}~\cite{strotz1955myopia} {({a.k.a} time-consistency~\cite{shapiro2009time,carpentier2012dynamic})}; from any point on an optimal trajectory, the remaining portion of the optimal trajectory is also optimal for the problem initiated at that point \cite{Weber_2016}. DP problems commonly have an additively separable objective function of the form $J(\mbf u, \mbf x)= \sum_{t=0}^{T-1} c_t(x(t),u(t)) + c_T(x(T))$. Problems of this form can be shown to satisfy the Principle of Optimality. However in problems such as optimal battery scheduling, we find non-additively separable objective functions. For example, if the objective is of the form $J(\mbf u, \mbf x)= \sum_{t=0}^{T-1} c_t(x(t),u(t)) + \max_{t_0 \le k \le T}{d_{{k}}(x({k}))}$ then the problem does not satisfy the Principle of Optimality. In this paper we propose a method for solving DP problems with non-separable objective functions by constructing equivalent DP problems with additively separable objective functions. Such reformulated problems then satisfy the Principle of Optimality and can therefore be solved using Bellman's Equation. {Moreover, we identify a class of problems, defined by Naturally Forward Separable (NFS) objectived, wherein state augmentation method does not substantially increase the complexity of the problem.}

For stochastic DP we generalize our framework and {propose a suitable definition of the \textit{Principle of Optimality}}. As discussed in \cite{Sniedovich_2002_eureka} such a definition is non trivial. Inspired by \cite{Hinderer_1970}, we construct probability measures on the sets the state variable can take at each time stage induced by the underlying random variables. We then say a stochastic problem satisfies the \textit{Principle of Optimality} if from any point on a trajectory followed using an optimal policy, $\pi$, the policy $\pi$ is also optimal for the problem initiated from that point with probability one.


Dynamic programming for problems which do not satisfy the Principle of Optimality has received relatively little attention. The only {general} approach to the problem seems to be that taken in~\cite{Duan} which considered the use of multi-objective optimization in the case where the objective function is ``backward separable''. Our approach differs from \cite{Duan} as we consider the class of ``forward separable'' objective functions. In this paper we show that almost any objective function is forward separable in a certain sense and that for such problems there exists an additively separable augmented-state dynamic programming problem that satisfies the Principle of Optimality and from which solutions to the original forward separable problem can be recovered - See Section~\ref{sec:ADP}. However, the resulting augmented-state DP problem has a higher dimensional state space than the original DP problem - an issue that can potentially render the augmented problem intractable due to the ``curse of dimensionality''. For this reason, we propose a complexity metric for the forward separable representation and show that in certain cases the dimensionality of the augmented system does not significantly exceed the dimensionality of the original problem - {a case where Bellman's equation can be used effectively~\cite{powell2007approximate} and which we refer to as \textit{Naturally Forward Separable} (NFS).}

{Note that although state augmentation has been used in the context of DP~\cite{lourens2012augmented,skaf2010shrinking}, the only tractable use of state augmentation to recover the Principle of Optimality appears to be~\cite{bauerle2013more} and~\cite{haskell2015convex}, who considered a DP problem with objective function of the form $J(\mbf u, \mbf x)= U(\sum_{t=0}^{T-1} c_t(x(t),u(t)))$; and~\cite{Sniedovich_1993}, who considered a DP problem with variance-type objective function. Both these results can be considered special cases of the NFS class of objective functions proposed in this paper.}


In practice it is rare to be able to analytically solve Bellman's Equation. Therefore, once the augmented-state DP problem is formulated we propose a map to an approximated DP problem that can be analytically solved. {Using an optimal solution from the approximated DP problem a feasible policy for the original problem is then constructed using a discretization scheme based on~\cite{Dufour_2012} and~\cite{Jones_2018}. {We show that as the number of discrete points increases, the resulting policies converge to the optimal policy}.}

{To illustrate the proposed methods, we consider battery scheduling for mitigating the effect of variability in renewable energy resources. Specifically, renewable energy sources are most accurately modeled as an uncontrollable Gauss-Markov (G-M) process and the battery (for both consumers and utilities) attempts to minimize energy costs based on time-of-use while also minimizing the maximum rate of energy consumption. Based on this model, we formulate the battery storage problem as a DP with a non-separable objective function consisting of both integrated time-of-use charges and a maximum term representing the demand charge. Furthermore, we propose a model of solar generation as a G-M process and minimize the expected value of the proposed objective.
The fundamental mathematical challenge with dynamic programming problems of this form is that, as shown in Section~\ref{sec:GDP}, problems which include maximum terms in the objective do not satisfy the \textit{Principle of Optimality} and thus the recursive solution of the Bellman equation (\cite{Bellman}) does not yield an optimal policy. To overcome this difficulty, we show that the battery scheduling problem is a special case of a forward separable DP problem with {an} NFS objective function. We then apply our state-augmentation technique to numerically solve the deterministic battery scheduling problem for given forecast solar data. In section \ref{sec:Solving the stochastic battery scheduling problem} we  apply our approach to the battery scheduling problem using a Gauss-Markov model of solar generation extracted from data provided by local utility SRP. Note that this result extends previous work which considered a more limited non-separable DP framework as applied to battery scheduling in~\cite{jones2017solving}.}

{Remarkably, almost no work has been done on optimal use of batteries for reduction of demand charges. The exceptions include the heuristic algorithms of~\cite{Malky} and the pioneering work of~\cite{braun_2006} , which considered \textit{only} a demand charge. Recently this group used an ad-hoc algorithm to consider a combined demand/consumption charge in~\cite{cai_2016} using detailed models of cooling/load. Furthermore, in \cite{Vijay} a similar energy storage problem is solved using optimized curtailment and load shedding. An $L_p$ approximation of the demand charge was used in combination with multi-objective optimization in~\cite{multi-objective} and, in addition, the optimal use of building mass for energy storage was considered in~\cite{Thermostat}, wherein a bisection on the demand charges was used. {We note that none of these approaches resolve the fundamental mathematical problem of DP with a non-separable cost function. 
}}

{The paper is organized as follows. {In Section~\ref{sec:GDP} we introduce a formal definition of the DP problem along with associated notation and use this framework to define a \textit{Principle of Optimality}.} Next, we consider a class of objective functions we refer to as forward separable. In Section~\ref{sec:ADP}, we show that for any DP problem with forward separable objective, there exists {an} augmented-state DP problem with separable objective for which the Principle of Optimality holds and from which solutions to the original DP with FS objective can be recovered. In Section \ref{sec: when we can use state augmentation} we define a class of objective functions, termed Naturally Forward Separable (NFS). We show DP problems with naturally forward separable objective functions can be tractable solved using state augmentation. In Section~\ref{sec:numerical} we show how to approximate and numerically solve augmented-state dynamic programming problems. Furthermore, we extend our framework to stochastic DP problems in Section~\ref{sec:SDP} and give a discretization scheme to solve stochastic DP's with additively separable objective functions in Section \ref{sec: numerically solving stochastic problems}. We summarize how state augmentation can be used with discretization methods to solve DP problems with NFS objective functions in Section \ref{sec: augmentation and discretization methods}. In Section~\ref{sec:battery1} we formulate and solve the battery scheduling problem as a DP with NFS objective function.}

  \section{Background: Dynamic Programming}\label{sec:GDP}
  In this paper, we propose a {framework for representing a general} class of Dynamic Programming (DP) problems. Specifically, we define a {general} DP problem as a sequence of optimization problems $\mcl G(t_0,x_0)$, indexed by $t_0\in \N$, and defined by an indexed sequence of objective functions $J_{t_0}: \R^{m \times (T-t_0)} \times \R^{n \times (T-t_0+1)} {\to \R}$ where we say that {the state and input sequence}, $\mathbf{u}^*\in \R^{m \times (T-t_0)}$ and $\mathbf{x}^*\in \R^{n \times (T-t_0+1)}$, solve $ \mcl G(t_0,x_0)$ if,
 \begin{align}
 &(\mathbf{u}^*,\mathbf{x}^*) {\in} \arg \min_{\mathbf u, \mathbf x} J_{t_0}(\mathbf u, \mathbf x) \label{eqn:opt}\\ \nonumber
 &\text{subject to:  }  \\ \nonumber
 & x(t+1)=f(x(t),u(t),t) \text{ for  } t=t_{0},..,T, \\ \nonumber
    & x(t_0)=x_0 , \text{ } x(t) \in X_t \subset \mathbb{R}^n \text{ for  } t=t_{0},..,T-1, \\ \nonumber
    & \text{ } u(t) \in U \subset \mathbb{R}^m \text{ for  } t=t_{0},..,T-1,\\ \nonumber
    &\mbf u=(u(t_0),...,u(T-1)) \text{ and } \mbf x =(x(t_0),...,x(T)),
 \end{align}
 where $f$ : $\mathbb{R}^n \times \mathbb{R}^m \times \mathbb{N} \to \mathbb{R}^n$, {$U \subset \R^m$ and $X_t \subset \R^n$} for all $t$. We denote $J_{t_0}^* =J_{t_0}(\mathbf u^*, \mathbf x^*)$.

 We call $\{x(t)\}_{t_0\le t \le T}$ the {set of} state variables and $n$ the state space dimension. Similarly we will call $\{u(t)\}_{t_0 \le t \le T-1}$ the input (control) variables and {$m=\dim( U )$} the input (control) space dimension. For cases where the dimension of the state variable, $x(t)$, varies with time, we slightly abuse notation and define the state space dimension as $\max_{t_0 \le t \le T} {\dim(X_t)}$.
 \begin{defn} \label{defn: additively seperbale function}
 	The function $J_{t_0}: \R^{m \times (T-t_0)} \times \R^{n \times (T-t_0+1)} {\to \R}$ is said to be additively separable if there exist functions, $c_T(x):\mathbb{R}^n \to \mbb R$, and $c_t(x,u):\mbb R^n \times \R^m \to \R $ for $t=t_0,\cdots T-1$ such that,
 	\begin{align} \label{eqn:additively seperable}
 	 J_{t_0}(\mbf u, \mbf x)= \sum_{t=t_0}^{T-1} c_t(x(t),u(t)) + c_T(x(T)),
 	\end{align}
 	where $\mbf u=(u(t_0),...,u(T-1)) \text{ and } \mbf x =(x(t_0),...,x(T))$.
 	\end{defn}
 	{To illustrate, we note that} the average {value of {a function $a_t: \R^n \to \R$}, defined as} $J(\mbf u, \mbf x)= \frac{1}{T} \sum_{t=0}^{T} a_t(x(t))$, is clearly an additively separable function. Variance type functions \eqref{eqn: var type functions}, however, are not additively separable.
\begin{defn}
     	We say the sequence of inputs  $\mathbf{u}=(u({t_0}),...,u({T-1})) \in \R^{m \times (T-t_0)}$ is feasible if $u(t) \in U \text{ for  } t=t_0,..,T-1$ and for $x(t+1)=f(x(t),u(t),t)$ and $x(t_0)=x_0$, then $x(t)\in X$ for all $t$. For a given $x$, we denote by $\Gamma_{t,x}$, the set $u \in U$ such that {$f(x,u,t) \in X_{t+1}$}. In this paper we only consider problems where $\Gamma_{t,x}$ is nonempty for all $x$ and $t$.
\end{defn}
Note that for this class of DP problems, feasibility is inherited. That is, if $\mathbf{u}=(u({t}),....,u({T-1}))$ is feasible with $\mathbf{x} =(x(t),\cdots,x(T))$ for $\mcl G(t,x(t))$ and  $\mathbf{v}=(v({s}),....,v({T-1}))$ is feasible with $\mathbf{h} =(h(s),\cdots,h(T))$ for $\mcl G(s,x(s))$ where $s>t$, then $\mathbf{w}=(u(t),\cdots,u(s-1),v({s}),....,v({T-1}))$ with $\mathbf{z} =(x(t),\cdots,x(s-1),h(s),\cdots,h(T))$ is feasible for $\mcl G(t,x(t))$.


\begin{defn} \label{def: policy}
       	A {(Markov)} \textit{policy} is any map from the present state and time to a feasible input $(x,t) \mapsto u(t) \in \Gamma_{x,t}$, as $u(t)=\pi(x,t)$. We denote the set of policies consistent with some DP problem as $\Pi$. We say that $\pi^*$ is an \textit{optimal policy} for Problem~\eqref{eqn:opt} if
       \[
\mathbf{u}^* =(\pi^*(x_0,t_0),....,\pi^*(x(T-1),T-1) )
       \]
where $x({t+1})^*=f(x(t)^*,\pi^*(x(t)^*,t),t)$ for all $t$.
\end{defn}
{We now define a ``Principle of Optimality'' consistent with our DP formulation and which provides a necessary condition} for such DP problems to be solvable using Bellman's equation \eqref{eqn:bellman}. If a DP problem is solvable using Bellman's equation, then this equation yields an optimal policy.
\begin{defn} \label{defn: principle of optimality deterministic}
We say a DP problem, $\mcl G(t_0,x_0)$, of the Form~(1) \textit{satisfies the Principle of Optimality} if the following holds. For any $s$ and $t$ with $t_0 \le t<s<T$, if $\mathbf u^*=(u(t),...,u(T-1))$ and $ \mathbf x^*=(x(t),...,x(T))$ solve $\mcl G(t,x(t))$ then  $\mathbf v=(u(s),...,u(T-1))$ and $ \mathbf h=(x(s),...,x(T))$ solve $\mcl G(s,x(s))$.
\end{defn}

The standard form of DP, equivalent to that defined in~\cite{Bellman}, solves indexed DP problems of the Form~\eqref{eqn:opt} with an additively separable objective function. We denote this class of DP problems by $\mcl P(t_0,x_0)$:
 \begin{align}
  &\min_{\mathbf u, \mathbf x}  J_{t_0}(\mathbf u, \mathbf x) = \sum_{t=t_0}^{T-1}{c_{t}(x(t),u(t))} + c_{T}({x(T)}) \label{eqn:DP}\\ \nonumber
 &\text{subject to:  } \\ \nonumber
& x(t+1)=f(x(t),u(t),t) \text{ for  } t=t_{0},..,T-1, \\ \nonumber
& x(t_0)=x_0 , \text{ } x(t) \in X_t \subset \mathbb{R}^n \text{ for  } t=t_{0},..,T, \\ \nonumber
& \text{ } u(t) \in U \subset \mathbb{R}^m \text{ for  } t=t_{0},..,T-1, \\ \nonumber
& \mbf u=(u(t_0),...,u(T-1)) \text{ and } \mbf x =(x(t_0),...,x(T)).
 \end{align}
 Note that $J_{T}(x)=c_{T}({x})$. We will refer to $x_0 \in \mathbb{R}^n$ as the initial state, $J_{t_0}$ is the objective function, $c_t$ : $\mathbb{R}^n \times \mathbb{R}^m \to \mathbb{R}$ for $t=t_0,..,T-1$, $c_T$ $\mathbb{R}^n \to \mathbb{R}$ are given functions and $f$ : $\mathbb{R}^n \times \mathbb{R}^m \times \mathbb{N} \to \mathbb{R}^n$ is a given vector field. The following lemma shows that this class of problems satisfies the proposed Principle of Optimality.
 \begin{lem} \label{lem: P(t,x) satisfies the principle of optimality}
Any problem of form $\mcl P(t_0,x_0)$ in \eqref{eqn:DP} satisfies the Principle of Optimality.
 \end{lem}
 \begin{proof}
 	Suppose $\mathbf u^*=(u(t),...,u(T-1))$ and $ \mathbf x^*=(x(t),...,x(T))$ solve $\mcl P(t,x(t))$ in (2). Now we suppose by contradiction that there exists some $s >t$ such that $\mathbf v=(u(s),...,u(T-1))$ and $ \mathbf h=(x(s),...,x(T))$ do not solve $\mcl P(s,x(s))$. We will show that this implies that $\mathbf u^*$ and $ \mathbf x^*$ do not solve $\mcl P(t,x)$ in (2), thus verifying the conditions of the Principle of Optimality.  If $\mathbf v$ and $ \mathbf h$ do not solve $\mcl P(s,x(s))$, then there exist feasible $\mathbf{w}$, $\mathbf z$ such that
 $J_{s}(\mathbf w, \mathbf z) < J_{s}(\mathbf v, \mathbf h)  $. i.e.
 	\begin{align*}
 	 J_{s}(\mathbf w,\mathbf z ) & = \sum_{t=s}^{T-1}{c_{t}(z(t),w(t))} + c_{T}({z(T)}) \\
 	& < \sum_{t=s}^{T-1}{c_{t}(x(t),u(t))} + c_{T}({x(T)})
 	=J_{s}(\mathbf v,\mathbf h )
 	\end{align*}
 	Now consider the proposed feasible sequences $\hat{\mathbf u}=(u(t),...,u(s-1), w(s),...,w(T-1))$ and $\hat{\mathbf x}=(x(t),...,x(s-1), z(s),...,z(T-1))$. It follows:
 	\begin{align*}
 	& J_{t}(\hat{\mathbf u},\hat{\mathbf x} ) \\
 	& =\sum_{k=t}^{s-1}{c_{k}(x(k),u(k))} + \sum_{k=s}^{T-1}{c_{k}(z(k),w(k))} + c_{T}({z(T)}) \\
 	& < \sum_{k=t}^{s-1}{c_{k}(x(k),u(k))}  + \sum_{k=s}^{T-1}{c_{k}(x(k),u(k))} + c_{T}({x(T)}) \\
 	& = J_{t}(\mathbf u^*,\mathbf x^* )
 	\end{align*}
 	which contradicts optimality of $\mbf u^*, \mbf x^*$. Therefore, this class of problems satisfies the Principle of Optimality.
 	\end{proof}

 \begin{prop}[\cite{Lerma_1996}] \label{prop:bellman}
 	For DP problems of the form $\mcl P(t,x)$ in \eqref{eqn:DP} with optimal objective, {$J_{t}^*=J_t(\mbf u^*, \mbf x^*)$}, define the function $F(x,t)= J_{t}^*$. Then the following holds.
 	  {	\begin{align} \label{eqn:bellman}
 	  	& F(x,t)=\inf_{u \in \Gamma_{t,x}}\{c_t(x,u)+F(f(x,u,t),t+1)\} \\ \nonumber
 	  	& \hspace{2.5cm}  \forall x \in X_t \text{ and }  \forall t \in \{t_0,..,T-1\}, \\
 	  	& F(x,T)=c_T(x) \qquad \forall x \in X_T \notag
 	  	\end{align} }
 \end{prop}
 Equation~\eqref{eqn:bellman} is often referred to as Bellman's equation and a function $F$ which satisfies Bellman's equation  is often referred to as the ``optimal {cost-to-go}'' function. Prop.~\ref{prop:bellman} shows that problems of the Form $\mcl P(t_0,x_0)$ define a solution to Bellman's equation which in turn indexes the optimal objective to the problem. Furthermore, for problems  $\mcl P(t_0,x_0)$, the solution to Bellman's equation can be obtained recursively backwards in time using a minimization on $u$. A solution to Bellman's equation provides a state-feedback law or optimal policy as follows.
\begin{cor}[\cite{Lerma_1996}]
Consider $\mcl P(t_0,x_0)$ in \eqref{eqn:DP}. Suppose $F(x,t)$ satisfies Equation~\eqref{eqn:bellman} for $\mcl P(t_0,x_0)$. Then if there exists a policy such that,
\[
\theta(x,t) {\in} \arg \min_{u \in \Gamma_{t,x}}\{c_t(x,u)+F(f(x,u,t),t+1)\},
\]
then $\theta$ is {an} optimal policy for the problem $\mcl P(t_0,x_0)$.
\end{cor}


\noindent \textbf{Dynamic Programming with Maximum Terms}
In this paper we consider the special class of indexed DP problems, denoted by $\mcl S(t_0,x_0)$ {and given in \eqref{eqn:S}}. In contrast to problems of the form $\mcl P(t_0,x_0)$ in~\eqref{eqn:opt}, class $\mcl S(t_0,x_0)$ has maximum terms in the objective. Specifically, these problems have the following form: \normalsize

{\small \begin{align}
   &\min_{ \mathbf u, \mbf x} J_{t_0} (\mbf u,\mbf x):=  \sum_{t=t_0}^{T-1}{c_{t}(x(t),u(t))} +c_{T}(x(T)) +\max_{t_0 \le k \le T}{d_{{k}}(x({k}))} \nonumber \\ \label{eqn:S}
 &\text{subject to:  }  \\ \nonumber
& x(t+1)=f(x(t),u(t),t) \text{ for  } t=t_{0},..,T-1, \\ \nonumber
& x(t_0)=x_0 , \text{ } x(t) \in X_t \subset \mathbb{R}^n \text{ for  } t=t_{0},..,T, \\ \nonumber
&  u(t) \in U \subset \mathbb{R}^m \text{ for  } t=t_{0},..,T-1, \\ \nonumber
& \mbf u=(u(t_0),...,u(T-1)) \text{ and } \mbf x =(x(t_0),...,x(T)),
   \end{align}
} \normalsize
where $c_T(x):\mathbb{R}^n \to \mbb R$; $c_t(x,u):\mbb R^n \times \R^m \to \R $ for $t_0 \le t \le T-1$; $d_t(x):\mathbb{R}^n \to \mbb R$ for $t=t_0,\cdots T$; $f$ : $\mathbb{R}^n \times \mathbb{R}^m \times \mathbb{N} \to \mathbb{R}^n$.	
   \begin{cex} \label{cex: max obj does not satisfy principl}
   	The class of DP problems of the form $\mcl S(t_0,x_0)$ in~\eqref{eqn:S} does not satisfy the Principle of Optimality.
   	\end{cex}
   	\begin{proof}
   		We give a counterexample. For $h > 0$, we consider the following problem $\mcl S(0,0)$:
    		\begin{align*}
    		&  \min_{\mathbf{u}\in \R^3, \mbf x\in \R^4} \quad {\sum_{t=0}^{2}{c_{t}(u({t}))} + \max_{0 \le k \le 3}{x(k)}}\\
    			& \text{subject to:  } x(t+1)=x(t)+u(t) \quad \forall t \in \{0,1,2\},  \\
    			&x(0)=0, \text{ } 0 \le x(t) \le h, \text{ } u(t) \in \{-h,0,h\},
    			\end{align*}  	
where here we define $c_{0}(u(0))=-u(0),$ $c_{1}(u(1))=u(1),$ $c_{2}(u(2))=-u(2)/2$.\\
Since $\mbf u \in \{-h,0,h\}^3$, there are 27 input sequences, only 8 of which are feasible. In Table~\ref{tab:counter}, we calculate the objective value of each feasible input sequence and deduce the optimal input is $\mbf{u}^*=(h,-h,h)$, yielding an optimal trajectory of $\mbf x^*=\{0,h,0,h\}$. Following this input sequence until $t=2$ we examine the problem $\mcl S(2,0)$.
    		\begin{align*}
    		& \min_{u(2) \in \R, 0 \le x(3) \le h }{c_{2}(u({2})) + \max_{2 \le k \le 3}{x(k)}}\\
    		& \text{subject to:  } x(t+1)=x(t)+u(t),  \\
    		& x(2)=0, \text{ } 0 \le x(t) \le h, \text{ } u(t) \in \{-h,0,h\}.	
    		\end{align*}
    For this sub-problem, there are two feasible inputs: $u(2)\in \{0,h\}$. Of these, the first is optimal (objective value $0$ vs $h/2$). Thus, although $\mbf u^*=\{h,-h,h\}$ and $\mbf x^*=\{0,h,0,h\}$ solve $\mcl S(0,0), \mbf v =\{h\}$ and $\mbf h =\{0,h\}$ do not solve $\mcl S(2,0)$.
%
%
   		\end{proof}
   		\begin{table}
   			\centering
   			\caption{This table shows the corresponding cost of each feasible policy used in the counter example in lemma 1}
   			\label{tab:counter}
   			\begin{tabular}{|l|l||l|l|l}
   				\cline{1-4}
   				feasible $\mbf u$  & objective value       & feasible $\mbf u$  & objective value   \\ \cline{1-4}
   				$(0,0,0)$  & 0 &  $(h,0,-h)$    & h/2  &  \\ \cline{1-4}
   				$(0,0,h)$    & h/2       & $(h,0,0)$ & 0  &  \\ \cline{1-4}
   				$(0,h,0)$ & 2h        & $(h,-h,0)$  & -h  &  \\ \cline{1-4}
   				$(0,h,-h)$  & (5/2)h & $(h,-h,h)$    & -(3/2)h  &  \\ \cline{1-4}
   			\end{tabular}
   		\end{table}

 \section{{Converting Forward Separable DP To Additively Separable DP}}\label{sec:ADP}

 In this section we define the class of forward separable objective functions. We will show that for dynamic programming problems with a forward separable objective function, augmenting the state variables allows us to use Bellman's equation to obtain an optimal policy.

 Forward separable functions were first defined in \cite{Envelope}. {Intuitively, this is the class of functions that can be separated into a {nested} composition of maps ordered forward in time.} In the next definition we build upon the concept of forward separability by introducing the notion of augmented dimension.
 \begin{defn}
 	The function $J: \R^{m \times (T-t_0)} \times \R^{n \times (T+1 -t_0)} \to \R $ is said to be forward separable if there exist {representation maps} $\phi_{t_0}: \R^n \times \R^m \to \R^{d_{t_0}}$, $\phi_T: \R^n \times \R^{d_{T-1}} \to \R$, and $\phi_i: \R^n \times \R^m \times \R^{d_{i-1}} \to \R^{d_{i}}$ for $i=t_0 + 1,\cdots T-1$ such that
  	\begin{align} \label{forward_sep_def}
	&J(\mbf u, \mbf x)=\phi_T( x(T),\phi_{T-1}(x(T-1),u(T-1),\phi_{T-2}(....,\\ \nonumber
	&\quad  \phi_{{t_0}+1}(x({t_0}+1),u({t_0}+1), \phi_{t_0}(x(t_0),u(t_0)))),....,))), \nonumber
 	\end{align}
where $\mbf u = (u(t_0),...,u(T-1)) \in \R^{m \times (T-t_0)}$ and $u(i) \in \R^m$ for $i \in \{t_0,...,T-1\}$; $\mbf x = (x(t_0),...,x(T)) \in \R^{n \times (T+1 -t_0)}$ and $x(i) \in \R^n$ for $i \in \{t_0,...,T\}$; $d_i \in \N$ for $i \in \{t_0,...,T-1\}$.

Moreover we say $J(\mbf u, \mbf x)$ is forward separable and has a representation dimension of $l$ if there exists $\{\phi_i \}$ that satisfies \eqref{forward_sep_def} and $l=\max_{i \in \{t_0,...,T-1\}} \{d_i\}$ where $d_i= \dim (Im\{\phi_i\})$.
 \end{defn}

\textbf{Note:}
	The representation dimension of a forward separable function is a property of the set $\{\phi_i \}$ chosen and not the function. The representation dimension of a forward separable function is not unique. {Moreover, the forward separable property of an objective function is independent of the DP problem it is associated with; forward separability is solely a property of the function only.}

Clearly, any additively separable objective function of the form
$
J(\mbf u, \mbf x) = \sum_{t=t_0}^{T-1} c_t(u(t),x(t)) + c_T(x(T))
$
is forward separable and has a representation dimension of 1 using,
\begin{align} \label{eqn: addative functions are forward sep}
& \phi_{t_0}(x,u)=c_{t_0}(x,u) \\ \nonumber
& \phi_i(x,u,w)=c_i(x,u)+w \quad \text{for } i=t_0+1,\cdots,T-1 \\ \nonumber
& \phi_T(x,w)=c_T(x)+w.
\end{align}
\subsection{How State Augmentation Solves Forward Separable DP Problems}
We now define the class of forward separable problems $\mcl H(t_0,x_0)$. Such problems are {a special case of} $\mcl G(t_0,x_0)$ \eqref{eqn:opt}, but not {of} $\mcl P(t_0,x_0)$ \eqref{eqn:DP}. Specifically, $\mcl H(t_0,x_0)$ has the form:
\begin{align}\label{opt:forward_sep}
&\min_{\mathbf u, \mbf x} J_{t_0}(\mbf u, \mbf x) \\ \nonumber
 &\text{subject to:  }  \\ \nonumber
& x(t+1)=f(x(t),u(t),t) \text{ for  } t=t_{0},..,T-1, \\ \nonumber
& x(t_0)=x_0 , \text{ } x(t) \in X_t \subset \mathbb{R}^n \text{ for  } t=t_{0},..,T, \\ \nonumber
& \text{ } u(t) \in U \subset \mathbb{R}^m \text{ for  } t=t_{0},..,T-1, \\ \nonumber
& \mbf u=(u(t_0),...,u(T-1)) \text{ and } \mbf x =(x(t_0),...,x(T)),
\end{align}
where $J_{t_0}$ is forward separable with associated representation maps $\phi_i$. For any forward separable DP problem $\mcl H(t_0,x_0)$, we may associate a new augmented-state DP problem {of form} $\mcl A(t_0,x_0)$, which is equivalent to $\mcl H(t_0,x_0)$ and which satisfies the Principle of Optimality. $\mcl A(t_0,x_0)$ is defined as
\small{ \begin{align} \label{augmentation}
	& \min_{ \mathbf u, \mbf z} L_{t_0} (\mathbf u, \mbf z) :=   \phi_{T}(z_1(T),z_2(T)) \\  \nonumber
	&	\text{subject to:  }  \\ \nonumber
	&\bmat{z_1(t+1)\\z_2(t+1)}= \bmat{f(z_1(t),u(t),t) \\ \phi_{t}(z_1(t),u(t), z_2(t)) } t_0+1 \le t < T-1\\ \nonumber
	&  \bmat{z_1(t_0+1)\\z_2(t_0+1)}= \bmat{f(z_1(t_0),u(t_0),T) \\ \phi_{t_0}(z_1(t_0),u(t_0)) }\\ \nonumber
	&  \bmat{z_1(t_0)\\z_2(t_0)} = \bmat{x_0\\ 0}, \text{ } z_{1}(t) \in X_t, \text{ } u(t) \in U \text{ for  } t=t_0,..,T \\ \nonumber
& \mbf u=(u(t_0),...,u(T-1)) \text{ and } \mbf z = \left(\bmat{ z_1(t_0) \\ z_2(t_0)} ,...,\bmat{z_1(T) \\ z_2(T)} \right) \nonumber
	\end{align} } \normalsize
where $f$ : $\mathbb{R}^n \times \mathbb{R}^m \times \mathbb{N} \to \mathbb{R}^n$, $z_1(t) \in \mathbb{R}^n$, $z_2(t) \in \mathbb{R}^{d_t}$, $d_t= \dim(Im\{\phi_{t-1}\})$ and $u(t) \in \mathbb{R}^m$ for all $t$.

\begin{lem} \label{lem: augmented optimization is equivalent to original}
	Suppose $J_{t_0}$ is the objective function for the DP problem $\mcl H(t_0,x_0)$ \eqref{opt:forward_sep} and is forward separable with associated representation maps $\phi_i$. Consider the augmented DP problem $\mcl A(t_0,x_0)$ \eqref{augmentation} and denote its objective function by $L_{t_0}$. Then $J^*_{t_0}=L^*_{t_0}$. Furthermore, suppose $\mbf u$ and $\mbf x=(x(t_0),...,x(T))$ solve $\mcl H(t_0,x_0)$  and $\mbf w$ and $\mbf z = \left(\bmat{ z_1(t_0) \\ z_2(t_0)} ,...,\bmat{z_1(T) \\ z_2(T)} \right)$ solve $\mcl A(t_0,x_0)$. Then $\mbf u=\mbf w$ and $x(t)=z_1(t)$ for all $t$.
\end{lem}
\begin{proof}
	Suppose  $\mbf w$ and $\mbf z$ solve $\mcl A(t_0,x_0)$. First we show that $\mbf w$ and $\mbf z_1 : = (z_1(t_0),...,z_1(T))$ are feasible for $\mcl H(t_0,x_0)$. Clearly $w(t) \in U$ for all $t$ and if we let $\mbf u = \mbf w$ then $x(0)=x_0$ and $x(t+1)=f(x(t),u(t),t)$ for all $t$. Likewise since $z_1(t_0)=x_0$ and $z_1(t+1)=f(z_1(t),u(t),t)$ for all $t$, we have $x(t)=z_1(t)\in X_t$ for all $t$. Hence $\mbf u$ and $\mbf x=\mbf z_1$ are feasible for $\mcl H(t_0,x_0)$.
	On the other hand, if $\mbf u$ and $\mbf x$ solve $\mcl H(t_0,x_0)$, then if we let $\mbf w=\mbf u$ and $\mbf z_1=\mbf x$ and define $z_2(t+1)=\phi_{t}(z_1(t),u(t), z_2(t))$, $z_2(t_0+1)=\phi_{0}(z_1(t_0),u(t_0))$, $z_2(t_0)=0$, then $\mbf w$ and $\mbf z$ are feasible. Furthermore, since $\mcl H(t_0,x_0)$ has a forward separable objective function we have,
	\begin{align*}
	&J_{t_0}(\mbf u, \mbf x)=\phi_T( z_1(T),\phi_{T-1}(z_1(T-1),w(T-1),\phi_{T-2}(....,\\ \nonumber
	&\quad  \phi_{{t_0}+1}(z_1({t_0}+1),w({t_0}+1), \phi_{t_0}(z_1(t_0),w(t_0)))),....,))). \nonumber
	\end{align*}
	However, we now observe
	\small{ \begin{align*}
		&z_2(T)= \phi_{T-1}(z_1(T-1),u(T-1),z_2(T-1)).\\
		&\vdots\\
		&z_2(t_0+1)= \phi_{t_0}(z_1(t_0),u(t_0)).\\
		&z_2(t_0)= 0.
		\end{align*} } \normalsize
	Hence we have,
	\small{
		\begin{align*}
		L_{t_0}(\mbf w, \mbf z) & = \phi_T(z_1(T),z_2(T)) \\
	&=\phi_T( z_1(T),\phi_{T-1}(z_1(T-1),w(T-1),\phi_{T-2}(....,\\ \nonumber
	&\quad  \phi_{{t_0}+1}(z_1({t_0}+1),w({t_0}+1), \phi_{t_0}(z_1(t_0),w(t_0)))),....,))) \\\nonumber
		&=J_{t_0}(\mbf u, \mbf x).
		\end{align*} } \normalsize
	Thus if $\mbf w$ and $\mbf z$ solve $\mcl A(t_0,x_0)$ with objective $L^*_{t_0}=\phi_T(z_1(T),z_2(T))$, then $\mbf w$ and $\mbf z_1$ solve $\mcl H(t_0,x_0)$ with objective value $J^*_{t_0}$.
\end{proof}

\begin{prop} \label{prop: augmentation satisfies the principle of optimality}
	The augmented DP problem $\mcl A(t_0,x_0)$ in \eqref{augmentation} satisfies the Principle of Optimality.
\end{prop}
\begin{proof}
	$\mcl A(t_0,x_0)$ is a special case of $\mcl P(t_0,x_0)$ \eqref{eqn:DP} where $c_i=0$ for $i \neq T$ and $c_T([z_1, z_2]^T)=\phi_T(z_1,z_2)$. Lemma \ref{lem: P(t,x) satisfies the principle of optimality} shows DP problems of the form $\mcl P(t_0,x_0)$ satisfy the Principle of Optimality.
\end{proof}

Lemma \ref{lem: augmented optimization is equivalent to original} tells us that for any DP problem with forward separable objective, $\mcl H(t_0, x_0)$ \eqref{opt:forward_sep}, there exists an equivalent DP problem of the form $\mcl A(t_0, x_0)$ \eqref{augmentation}. Furthermore Proposition \ref{prop: augmentation satisfies the principle of optimality} shows that $\mcl A(t_0, x_0)$ satisfies the Principle of Optimality. Therefore a solution for $\mcl H(t_0, x_0)$ can be found by recursively solving Bellman's equation \eqref{eqn:bellman} for $\mcl A(t_0, x_0)$.

To understand the augmented approach intuitively, we note that DP breaks a multi-period planning problem into simpler sub-problems at each stage. However, for non-separable problems, to make the correct decision at each stage we need past information about the system. In this context, the augmented state contains the information from the trajectory history necessary to make the correct decision at the present time. However by adding augmented states we increase the state space dimension and the complexity of the DP problem.

\begin{cor} \label{cor: the dimensions of the augmented problem}
	Suppose the forward separable function, $J: \R^{m \times (T-t_0)} \times \R^{n \times (T+1 -t_0)} \to \R $, is the objective function for DP problem $\mcl H(t_0, x_0)$ \eqref{opt:forward_sep} and has a representation dimension of $l$. Then the associated augmented DP problem with this representation,  $\mcl A(t_0, x_0)$ \eqref{augmentation}, has a state space of dimension $l+n$ and input space of dimension $m$.
\end{cor}
\begin{proof}
	{From the definition of $\mcl A(t_0, x_0)$ \eqref{augmentation}, the state space dimension is $n+\max_{t_0 \le t \le T}$, where $d_t= \dim(Im\{\phi_{t-1}\})$. From the definition of representation dimension, we have $\max_{t_0 \le t \le T}d_t= l$  and hence it follows that the state space dimension is $n + \max_{t_0 \le t \le T} d_t = n+l$.  }
\end{proof}

\section{{A Class Of DP For Which The Use Of State Augmentation Is Tractable }} \label{sec: when we can use state augmentation}
{It is well known that discretization of the state space combined with a solution of Bellman's equation {become computationally intractable} when the discretized dimension increases; this is often called ``the curse of dimensionality''. In the previous Section, we proved that any non-separable DP of state space dimension $n$ can be converted to a separable augmented DP with state-space dimension $n+l$, where $l$ is the representation dimension of the objective function. However, for some representations, $l$ may increase as the time interval increases - thus triggering the curse of dimensionality. To address this problem, in this section, we define a class of forward separable objective functions, called Naturally Forward Separable (NFS) functions, with representation dimension, $l$, which is independent of the number of time steps and the dimension of the state and input space.   }

{Before we define NFS functions we motivate this new class of functions by showing that it is possible to represent any function as a forward separable function}. To do this we introduce some additional notation. Specifically, for a vector $v=(v_1,...,v_n)^T \in \R^n$ we define $[v]_{i}^j=(v_i,...,v_j)$ for $1 \le i <j \le n$.

\begin{lem} \label{lem: any function is forward seperable}
	Any function $J: \R^{m \times (T-t_0)} \times \R^{n \times (T+1 -t_0)} \to \R $ is forward separable with a representation of dimension $l(n,m,T -t_0)=(T-t_0)(n+m)$.
\end{lem}
\begin{proof}
Consider a function $J: \R^{m \times (T-t_0)} \times \R^{n \times (T+1 -t_0)} \to \R $. To show $J$ is forward separable we define a forward separable representation $\{ \phi_i\}_{i=t_0}^T$ which satisfy \eqref{forward_sep_def} as follows.

First, define $\phi_{t_0}: \R^n \times \R^m \to \R^{n+m}$ as
 \begin{align*}
 \phi_{t_0}(x,u)=[x^T , u^T] = \bmat{x_1,...,x_n,  u_1,...,u_m}.
\end{align*}

For $i \in \{t_0 +1,...T-1\}$ the define $\phi_i: \R^n \times \R^m \times \R^{(i-t_0)(n+m)}\to \R^{(i+1 -t_0)(n+m)}$ as
\begin{equation*}
\phi_i(x,u,w) =\bmat{[w]_1^{n(i-t_0)}, x^T, [w]_{n(i-t_0)+1}^{(i-t_0)(n+m)}, u^T}.
\end{equation*}

Lastly, define $\phi_T:\R^{n} \times \R^{(T-t_0)(n+m)} \to \R $ as
\begin{equation*}
\phi_T(x,w)=J([[w]_1^{n(T-t_0)},x],[w]_{n(T-t_0)+1}^{(n+m)(T-t_0)}).
\end{equation*}
Clearly, this definition of $\phi_i$ satisfies~\eqref{forward_sep_def}. Furthermore, it can be seen that the maximum dimension of the images of the maps $\{ \phi_i\}_{i=t_0}^T$ is $(T-t_0)(n+m)$ showing the dimension of this representation of $J$ is $l(n,m,T-t_0)=(T-t_0)(n+m)$.
\end{proof}


In the above approach to show that $J(\mbf u, \mbf x)$ is forward separable we naively took the strategy of using the functions $(\phi_i)_{t_0 \le i \le T}$ to act like memory functions; that is to store the entire historic state trajectory and input sequence used. If $J(\mbf u, \mbf x)$ is the objective function for some DP problem of form  $\mcl H(t_0, x_0)$ \eqref{opt:forward_sep} then this approach would result in the associated augmented DP problem,  $\mcl A(t_0, x_0)$ \eqref{augmentation}, having a very large state space dimension. Specifically, Corollary \ref{cor: the dimensions of the augmented problem} shows that $\mcl A(t_0, x_0)$ has state space dimension $(T-t_0)(n+m) +n$. Clearly, for a large number of time-steps, $T-t_0$, $\mcl A(t_0, x_0)$ is intractable. For this reason we next define a special class of forward separable functions that have a representation with dimension independent of the number of time-steps.


\begin{defn} \label{def: independent augmented dimension}
	We say a function $J: \R^{m \times (T-t_0)} \times \R^{n \times (T+1 -t_0)} \to \R $ is a \textit{Naturally Forward Separable} (NFS) function if there exists maps, $\{ \phi_i\}_{i=t_0}^T$, that satisfy \eqref{forward_sep_def} with representation dimension independent of $n$, $m$ and $T$.
\end{defn}

\subsection{{An Algebra Of Naturally Forward Separable Functions}} \label{subsec: algebra of NFSF's}

{Given a function, $J: \R^{m \times (T-t_0)} \times \R^{n \times (T+1 -t_0)} \to \R$, there is no obvious way to determine whether $J$ is NFS. Instead, in this section, we show that the set of NFS functions form an algebra closed under pointwise multiplication and which is preserved under nonlinear transformation - implying that simple NFS functions (`building blocks') can be combined to construct new, more complex, NFS functions. In this way, one might approach the problem of finding representation maps for a function, $J$, by combining known NFS ``building blocks''. Several examples of such ``building blocks'' can be found in Subsection~\ref{sec: Examples of Natutally Forward Seperable Functions}. We first prove closure under addition and pointwise multiplication.}

\begin{lem} \label{lem: addition of NFSF is NFSF}
	{Consider the function $U: \R \to \R$ and the NFS function, $J_1: \R^{m_1 \times (T_1-t_1)} \times \R^{n_1 \times (T_1+1 -t_1)} \to \R $ and $J_2: \R^{m_2 \times (T_2-t_2)} \times \R^{n_2 \times (T_2+1 -t_2)} \to \R $, with representation dimensions $l_1$ and $l_2$ respectively. The functions $G_1(\mbf u, \mbf x)= J_1(\mbf u, \mbf x) + J_2(\mbf u, \mbf x)$, $G_2(\mbf u, \mbf x)=J_1(\mbf u, \mbf x) \cdot J_2(\mbf u, \mbf x)$ and $G_3(\mbf u, \mbf x)=U\left(J_1(\mbf u, \mbf x) \right)$ are NFS functions with representation dimension less than or equal to $l_1+l_2$, $l_1+l_2$, and $l_1$, receptively}.
\end{lem}

\begin{proof}
	{	For simplicity let us consider the case where $t_1 = t_2$ and $T_1 = T_2$; other cases follow by the same argument. As $J_1$ and $J_2$ are forward separable functions there exist associated representations $\{g_i\}$ and $\{h_i\}$ such that $J_1$ and $J_2$ can be written {in} the form \eqref{forward_sep_def} and with associated representation dimensions $l_1$ and $l_2$, respectively. We now show that $G_1$ is forward separable by defining the associate representation $\{\phi_i \}$ such that $G_1$ can be written {in} the form \eqref{forward_sep_def}. Specifically, let
	\begin{align} \label{eqn:two_forward_sep}
	& \phi_{t_1}(x,u)= \bmat{g_{t_1}(x,u)\\h_{t_1}(x,u)},\\ \nonumber
	& \phi_{i}(x,u,w)= \bmat{g_{i}(x,u,[w]_1^{d_{i-1}})\\h_{i}(x,u,[w]_{d_{i-1} +1}^{d_{i-1} + s_{i-1}})} \text{ for } i \in \{t_1 +1,....,T_1-1\}\\ \nonumber
	& 	\phi_{T_1}(x,w) = g_T(x,[w]_{1}^{d_{T_1 -1}}) + h_T(x,[w]_{d_{T_1 -1}+1}^{d_{T_1 -1} + s_{T_1 -1} }),
	\end{align}
	where $d_i= \dim(Im\{g_{i}\}) $ and $s_i = \dim(Im\{h_{i}\})$ for $i \in \{t_1,...,T_1-1\}$.}
	
	{We conclude that $G_1$ has a representation dimension, denoted $l_{G_1}$, such that
	\begin{align*}
	l_{G_1} & = \max_{i \in \{t_1,...,T_1-1\}} \{d_i + s_i\} \\
	& \le \max_{i \in \{t_0,...,T-1\}} \{d_i\} + \max_{i \in \{t_0,...,T-1\}} \{s_i\} \\
	& = l_1+ l_2.
	\end{align*}
	Furthermore, by a similar argument it can be shown that $G_2$ and $G_3$ are NFS with representation dimension less than or equal to $l_1+ l_2$. We are able to show this using the same representation maps $\{\phi_i\}_{t_1 \le i \le T_1 -1}$ from \eqref{eqn:two_forward_sep} with the terminal representation map for $G_2$ given by
\[
\phi_{T_1}(x,w) = g_T\left(x,[w]_{1}^{d_{T_1 -1}}\right)\cdot h_T\left(x,[w]_{d_{T_1 -1}+1}^{d_{T_1 -1} + s_{T_1 -1} }\right),
\]
and the terminal representation map for $G_3$ given by
\[
\phi_{T_1}(x,w) = U\left(g_T\left(x,[w]_{1}^{d_{T_1 -1}}\right)\right).
\] }	
\end{proof}

\subsection{{Simple Examples Of NFS Functions}} \label{sec: Examples of Natutally Forward Seperable Functions}
{The first example of a NFS function is found in problems involving risk measures and certainty equivalents~\cite{bauerle2013more}. In this case, we have the function $U(x)= \frac{1}{\gamma}e^{\gamma x}$ and apply the following:}
\begin{ex} \label{ex: utility function}
	{For any functions $U: \R \to \R$ and $c_t:\R^n\times \R^m\rightarrow \R$,
	\[
	J(\mbf u, \mbf x)= U \left( \sum_{t=t_0}^{T-1}{c_{t}(x(t),u(t))} \right)
	\]
is NFS with representation dimension 1.
}
\end{ex}
\begin{proof}
	{The additively separable function $\sum_{t=t_0}^{T-1}{c_{t}(x(t),u(t))}$ is NFS using the representation maps given in \eqref{eqn: addative functions are forward sep}. It therefore follows $J$ is NFS by Lemma \ref{lem: addition of NFSF is NFSF}.}
\end{proof}
\begin{ex}
{	The mixed p-norm function given by
	\[
	J(\mbf u, \mbf x)= \sum_{j=1}^N \left( \sum_{t=t_0}^{T-1}{c_{j,t}(x(t),u(t))^{p_j}} \right)^{\frac{1}{p_j}},
	\]
	where $p_j>0$ for all $j \in \{1,...,N\}$, $c_{j,t}:\R^n\times \R^m\rightarrow \R^+$, and $N \in \N$, is NFS with representation dimension $N \in \N$.}
\end{ex}

\begin{proof}
	Follows as $J(\mbf u, \mbf x)$ that can be written {in} the form of~\eqref{forward_sep_def} using maps $\phi_{t_0}(x,u)= [c_{1,t_0}(x(0),u(0))^{p_1},...,c_{N,t_0}(x(0),u(0))^{p_N}]^T$, $\phi_{i}(x,u,w)= [c_{1,i}(x(i),u(i))^{p_1} +w_1 ,...,c_{N,i}(x(i),u(i))^{p_N} + w_N]^T$ for all $i \in \{t_0+1,...,T-1\}$, and $\phi_{T}(x,w)= \sum_{j=1}^N w_j^{\frac{1}{p_j}}$.
\end{proof}


{We next give a NFS function that can be considered a discrete time version of the Green measure; when used as an objective function for a DP problem it measures the amount of time the state and input spend in some set.  }

 \begin{ex} \label{NFSF 1} Consider the function $ J: \R^{m \times (T-t_0)} \times \R^{n \times (T+1-t_0)} \to \R$ defined as
{	\[
 	J(\mbf u, \mbf x)=| \{ i\in \{t_0,...,T\} : (x(i),u(i)) \in S \} |
 	\] }
 	where $\mbf u=(u(t_0),...,u(T-1))$, $u(t) \in \R^m$, $\mbf x =(x(t_0),...,x(T))$, $x(t) \in \R^n$, $S \subset \R^n \times \R^m$ and for $B \subset \N$ we denote $| B|$ to be the cardinality of the set $B$. Then $J$ is NFS and has a representation of dimension 1.
 \end{ex}
\begin{proof}
	We present functions such that $J(\mbf u, \mbf x)$ that can be written {in} the form of~\eqref{forward_sep_def}.
	
	Define $\phi_{t_0}: \R^n \times \R^m \to \R$ as
	\[ \phi_{t_0}(x,u)= \begin{cases}
	1 & \text{if } (x,u) \in S \\
	0 & \text{otherwise}
	\end{cases}.
		\]

Define $\phi_t: \R^n \times \R^m \times \R \to \R$ for  $1 \le t \le T -1$ as

	\[ \phi_t(x,u, w)= \begin{cases}
w + 1 & \text{if } (x,u) \in S \\
w & \text{otherwise.}
\end{cases} \]
	
Define the function $\phi_T: \R^n \times \R \to \R$ as

\[\phi_T(x, w)= \begin{cases}
w + 1 & \text{if } (x,u) \in S \\
w & \text{otherwise.}
\end{cases}\]
This definition of $\phi_i$ satisfies~\eqref{forward_sep_def}. Moreover it can be seen that the maximum dimension of the images of the maps $\{ \phi_i\}_{i=t_0}^T$ is $1$ implying that the dimension of this representation of $J$ is 1.
 \end{proof}
 \begin{ex} \label{NFSF 2} Consider the variance type function,  $ J: \R^{m \times T} \times \R^{n \times (T+1)} \to \R$ defined as
 	\begin{equation} \label{eqn: var type functions}
 	 	J(\mbf u, \mbf x)= \sum_{t=0}^{T} \left[ a_t(x(t)) - \frac{1}{T} \sum_{s=0}^{T} a_s(x(s))   \right]^2
 	\end{equation}
 	where  $\mbf u=(u(0),...,u(T-1))$, $u(t) \in \R^m$, $\mbf x =(x(0),...,x(T))$, $x(t) \in \R^n$, and $a_t:\R^n \to \R$. Then $J$ is NFS and has a representation dimension of 2.
 \end{ex}
 \begin{proof}
 	Expanding the right hand side of \eqref{eqn: var type functions} as in \cite{Sniedovich_1993} we get,
 	\small{
 	\begin{align*}
& J(\mbf u, \mbf x)  \\
&=\sum_{t=0}^{T} \left[ a_t^2(x(t)) - \frac{2}{T}a_t(x(t))\sum_{s=0}^{T} a_s(x(s)) + \frac{1}{T^2} \left(\sum_{s=0}^{T} a_s(x(s))\right)^2 \right] \\
 & =  \sum_{t=0}^{T} a_t^2(x(t)) - \frac{1}{T} \left[ \sum_{s=0}^{T} a_s(x(s)) \right]^2.
 	\end{align*} } \normalsize
 	We now present functions $J(\mbf u, \mbf x)$ that can be written {in} the form of~\eqref{forward_sep_def}.
 		We define $\phi_{t_0}: \R^n \times \R^m \to \R^2$ as
 	 	\begin{equation}
 	 	\phi_0(x,u) = \bmat{a_1^2(x) \\a_1(x)}.  \nonumber
 	 	\end{equation}
 		We define $\phi_{i}: \R^n \times \R^m \times \R^2 \to \R^2$ as
 	\begin{equation}
 	\phi_i(x,u,[w_1,w_2]^T)= \bmat{w_1 +a_i^2(x) \\ w_2  +a_i(x)} \text{ for } 1\le i \le T-1. \nonumber
 	\end{equation}
 	 		Finally, $\phi_{T}: \R^n \times \R^2 \to \R$ is given by,
 	\begin{equation*}
 	\phi_T(x,[w_1,w_2]^T)= (w_1 + a_T^2(x)) - \frac{1}{T} \left( w_2+ a_T(x) \right)^2.
 	\end{equation*}
This definition of $\phi_i$ satisfies~\eqref{forward_sep_def}. Moreover it can be seen that the maximum dimension of the images of the maps $\{ \phi_i\}_{i=t_0}^T$ is $2$ showing the dimension of this representation of $J$ is 2. \end{proof}

 We now show that the maximum function, that appears in the objective function of the battery scheduling problem in Section \ref{sec:battery1}, is NFS.
 \begin{ex} \label{lem: sup forward sep}
 	Consider the function $ J: \R^{m \times T} \times \R^{n \times (T+1)} \to \R$ such that,
 	\[
 J(\mbf u, \mbf x)=\max\{\max_{0 \le k \le T-1}\{c_k(u(k),x(k))\}, c_T(x(T))\}
 \]
 where $\mbf u=(u(0),...,u(T-1))$, $u(t) \in \R^m$, $\mbf x =(x(0),...,x(T))$, $x(t) \in \R^n$, $c_k: \R^m \times \R^n \to \R$ for $0 \le k \le T-1$ and $c_T:\R^n \to \R$. Then $J$ is NFS and has a representation dimension of 1.
 \end{ex}
 \begin{proof} \small{ \begin{align*}
 	J(\mbf u, & \mbf x) = \max\{\max_{0 \le k \le T-1}\{c_k(u(k),x(k))\}, c_T(x(T))\}\\
 	& = \max\{c_T(x(T)),\max\{c_{T-1}(u(T-1),x(T-1)),\cdots\\
 &\qquad \max\{..,\max\{c_1(u(1),x(1)),\max\{c_0(u(0),x(0))\}\},..\}\}.
 	\end{align*} } \normalsize
It is now clear we can write $J$ {in} the form of \eqref{forward_sep_def} as follows. The function $\phi_{t_0}: \R^n \times \R^m \to \R$ is defined by,
	\begin{equation*}
	\phi_{t_0}(x,u)=c_{t_0}(x,u).
	\end{equation*}
	
The function $\phi_{i}: \R^n \times \R^m \times \R \to \R$ is defined by,
	\begin{equation*}
	\phi_i(x,u,w)=\max(c_i(x,u),w) \text{ for } t_0 +1 \le i \le T-1.
	\end{equation*}
	
The function $\phi_{T}: \R^n \times \R \to \R$ is defined by,
\begin{equation*}
\phi_T(x,w) = \max (c_T(x),w).
\end{equation*}

This definition of $\phi_i$ satisfies~\eqref{forward_sep_def}. Moreover it can be seen that the maximum dimension of the images of the maps $\{ \phi_i\}_{i=t_0}^T$ is $1$ showing the dimension of this representation of $J$ is 1. \end{proof}

\section{{Solving Deterministic Additively Separable DP Problems}} \label{sec:numerical}
In Section \ref{sec:ADP} we showed that all forward separable problems of the form $\mcl H(t_0,x_0)$ have an equivalent DP problem of the form $\mcl A(t_0,x_0)$. Problems of the form $\mcl A(t_0,x_0)$ are special cases of problems of the form $\mcl P(t_0,x_0)$. In this section we address the problem of implementation by numerically solving problems of the form $\mcl P(t_0,x_0)$.

For implementation, we use an approximation scheme that maps our class of DP problems to a much simpler class of DP problems with finite state and control spaces. It is known for DP problems with finite state and control spaces that the infimum in Bellman's equation \eqref{eqn:bellman} is attained and the optimal cost to go function, $F(x,t)$, can be computed by enumeration. Similar numerical schemes with convergence proofs can be found in \cite{Jones_2018} and~\cite{Dufour_2012}.

\subsection{Construction Of Approximated Tractable DP Problems}
 Consider the DP problem $\mcl P(t_0,x_0)$ \eqref{eqn:DP} with compact state and control spaces of the form $X=[\underline{x},\bar{x}]^n$ and $U=[\underline{u},\bar{u}]^m$. For DP problems of this form it is not generally possible to solve Bellman's Equation \eqref{eqn:bellman}. We thus need to consider a sequence of ``close'' DP problems with countable state and control spaces. We define a sequence of approximated DP problems indexed by $k$ and denoted by $\mcl P_k(t_0,x_0)$,
 \begin{align}
 &\min_{\mathbf u, \mathbf x} J_{t_0}(\mathbf u, \mathbf x)= \sum_{t=t_0}^{T-1}{c_{t}(x(t),u(t))} + c_{T}({x(T)}) \label{eqn: approx opt}\\ \nonumber
 &\text{subject to:  }  \\ \nonumber
  & x(t+1)=argmin_{y \in X_k}\{||y-f(x(t),u(t),t)||_2\},  \\ \nonumber
 & x(t_0)=x_0 , \text{ } x(t) \in {X}_k \subset \mathbb{R}^n, \text{ } u(t) \in {U}_k \subset \mathbb{R}^m \text{ for  } t=t_{0},..,T, \\ \nonumber
 & \mbf u=(u(t_0),...,u(T-1)) \text{ and } \mbf x =(x(t_0),...,x(T)),
 \end{align}
 where ${X}_k=\{x_1,...,x_k\}^n$ such that $\underline{x}=x_1<x_2<...<x_k=\bar{x}$ and $||x_{i+1}-x_{i}||_2=\frac{\bar{x}-\underline{x}}{k}$ for $1 \le i \le k-1$, and ${U}_k=\{u_1,...,u_k\}^m$ such that $\underline{u}=u_1<u_2<...<u_k=\bar{u}$ and $||u_{i+1}-u_{i}||_2=\frac{\bar{u}-\underline{u}}{k}$ for $1 \le i \le k-1$.

{Note that our approximation scheme is based on the discretization of the state and input space and as such is subject to the ``curse of dimensionality". However, discretization-based schemes for solving Bellman's equation have efficient parallel implementations - see~\cite{maidens2016parallel}. Alternatively, other Approximate Dynamic Programming (ADP) schemes such as \cite{rust1997using} use grid sampling and can be shown to converge with respect to expected cost.}



\subsection{Constructing A Feasible Policy From The Solution Of The Approximated DP Problem}

By iteratively solving Bellman's equation \eqref{eqn:bellman} we can find an optimal solution to $\mcl P_k(t_0,x_0)$ which we denote as $(\mbf x^*_k, \mbf u^*_k)$. Because the vector fields that define the underlying dynamics of $\mcl P(t_0,x_0)$ and $\mcl P_k(t_0,x_0)$ are different, the solution $(\mbf x^*_k, \mbf u^*_k)$ is not necessarily feasible for $\mcl P(t_0,x_0)$. However using an optimal policy for $\mcl P_k(t_0,x_0)$, $\pi^*_k$, we can construct a feasible policy for $\mcl P(t_0,x_0)$ in the following way,
\begin{equation} \label{eqn:constructed policy}
\theta_k(x,t)= \arg \min_{u \in \Gamma_{t,x} } ||\pi^*_k( \arg \min_{y\in {X_k}}\{||y-x||_2\},t) - u ||_2 \in \Pi
\end{equation}
where we recall $\Gamma_{t,x}$ is the set of feasible inputs such that if $u \in \Gamma_{t,x}$ then $u \in U$ and $f(x,u,t) \in X$ for the DP problem $\mcl P(t_0,x_0)$ \eqref{eqn:DP}.

\subsection{Convergence Of Our Constructed Policy}
Suppose $\theta_k(x,t)$, from \eqref{eqn:constructed policy}, is a feasible policy for $\mcl P(t_0,x_0)$ constructed from {an} optimal policy of $\mcl P_k(t_0,x_0)$ using \eqref{eqn:constructed policy}. Let $\mathbf u_k = (\theta_k(x_0,t),...,\theta_k(x_k(T-1),T-1) )$ and $ \mathbf x_k = (x_k(t_0),...,x_k(T))$ where $x_k(t_0)=x_0$, $x_k(t+1)=f(x_k(t),\theta_k(x_k(t),t),t)$ and $f$ is the vector field from $\mcl P_k(t_0,x_0)$. {From Theorem 2 in \cite{Jones_2018} if $\mcl P(t_0,x_0)$ satisfies certain continuity assumptions} then it is known that 
\begin{equation} \label{eqn:convergance conjecture} \lim_{k \to \infty} | J_{t_0}(\mathbf u_k, \mathbf x_k) - J_{t_0}^*|=0,
\end{equation} where $J_{t_0}(\mathbf u_k, \mathbf x_k) $ is the resulting value objective function of $\mcl P(t_0,x_0)$ when the policy $\theta_k$ is used and $J_{t_0}^*$ is the optimal value of the objective function.
\section{Stochastic DP And Optimal Policies} \label{sec:SDP}
%

{As shown in Counterexample \ref{cex: max obj does not satisfy principl} there exist non-separable deterministic DP problems that do not satisfy the Principle of Optimality as formulated in Definition \ref{defn: principle of optimality deterministic}. This Principle Of Optimality stated that the optimality of an input sequence for any instantiation of the sequence of DP problems is inherited by every subsequent instantiation - implying that recursive use of Bellman's equation will eventually return an input which solves the DP problem for every instantiation.}

In the stochastic case, however, additively separable DP problems may not satisfy this version of Principle of Optimality. Specifically, we show that even in the case of an additively separable objective function, stochastic perturbations can drive the system to a state wherein the original optimal input sequence is no longer optimal. This implies that for problems which include stochastic perturbations, a new criterion must be proposed for the use of Bellman's equation. This reformulation of the Principle Of Optimality is non-trivial and requires us to re-work our definitions of DP. Specifically, we formulate the stochastic DP problem as a time-indexed sequence of DP problems, wherein the variable is not the input sequence, but is rather replaced by the policy. This policy is then used to generate a \textit{set} of possible trajectories and input sequences, indexed by the set of possible instantiations of the random variables.

{Using this new formulation of the stochastic DP problem, our definition of the Principle Of Optimality requires that optimality of the \textit{policy} is inherited by every possible instantiation of the trajectory with probability one - i.e. the policy may be sub-optimal, but only on a set of instantiations of measure zero. We then show that the state-augmentation strategy we propose for DP problems with stochastic perturbations and NFS objective functions results in a stochastic DP problem for which this revised the Principle Of Optimality holds. Finally, we show how Bellman's equation can be used in this case to recover an optimal policy.}

%


{We begin by defining the map from a chosen policy ($\pi$), initial condition ($x_0$) and instantiation of the random variables ($[\mbf v]_{t_0}^{T-1}$) to the resulting trajectory, $x$; this will allow us to state precisely which random variables the expectation in the objective function is respect to. For simplicity we only consider random variables with Gaussian distributions.}

\begin{defn} \label{defn: the trajectory map}
	For a vector field $f$ : $\mathbb{R}^n \times \mathbb{R}^m \times \mathbb{N} \times \R^q \to \mathbb{R}^n$, a set of polices $\Pi$ associated with some DP problem, a starting time $t_0 \in \N $, and terminal time $T \in \N$, let us denote \textbf{the state map} by $\psi_{f,t_0}:\Pi \times \R^n \times \N \times \R^{q \times (T-t_0)} \to \R^n$. We say that $x=\psi_{f,t_0}(\pi, x_0, T, [\mbf v]_{t_0}^{T-1})$ if $x=x(T)$ where $x(T)$ is a solution to the following recursion equations $x(t_0)=x_0$, $x(t+1)=f(x(t),\pi(x(t),t),t,v(t))$ for $t \in \{t_0,..., T-1\}$ and $[\mbf v]_{t_0}^{T-1} =[v(t_0),..,v(T-1)] \in \R^{q \times (T-t_0)} $. We denote the image of the state vector under a set of instantiations $ Y \subset \R^{q \times (T-t_0)}$ by $\psi_{f,t_0}(\pi, x_0, T, Y)= \{ \psi_{f,t_0}(\pi, x_0, T, [\mbf v]_{t_0}^{T-1}) \in \R^n: [\mbf v]_{t_0}^{T-1} \in Y  \}$.
	
	We also denote \textbf{the trajectory map} by 	$\Phi_{f,t_0}: \Pi \times \R^n \times \N \times \R^{q \times (T-t_0)} \to \R^{m \times (T-t_0)}  \times \R^{n \times (T-t_0 +1)} $. We say that $(\mbf u, \mbf x)=\Phi_{f,t_0}(\pi, x_0,T, [\mbf v]_{t_0}^{T-1})$ if $\mbf u = (\pi(x(t_0),t_0),...,\pi(x(T-1),T-1))$, and $\mbf x =(x(t_0),...,x(T))$ is such that $x(t)=\psi_{f,t_0}(\pi, x_0, t, [\mbf v]_{t_0}^{T-1})$ for $t \in \{t_0,..., T-1\}$.
\end{defn}
\normalsize We define the class of stochastic DP problems with forward separable objective as $\mcl H_s(t_0,x_0)$, \normalsize

{ \small \begin{align}
	&\mathbf \pi^{\mcl H_s*} = \arg \min_{\pi \in \Pi} \;\; \mbb E_{ [\mbf v]_{t_0}^{T-1} }\left( J_{t_0}^{\mcl H_s}(\Phi_{f,t_0}(\pi, x_0, T, [\mbf v]_{t_0}^{T-1}))  \right) \label{eqn:Forward sep DP stochastic}\\ \nonumber
	&\text{subject to:  }  \text{ } \psi_{f,t_0}(\pi, x_0, t, [\mbf v]_{t_0}^{t-1}) \in X_t \text{ for  } t=t_0,..,T  \\ \nonumber
	&  \pi(x,t) \in U_t \text{ and }   v(t) \in \R^q \sim \mathcal{N}(\mbf 0,I_{q \times q}) \; \forall x \in X_t, \forall t=t_0,..,T-1,
	\end{align} }
where $J_{t_0}^{\mcl H_s}: \R^{m \times(T-t_0) } \times \R^{n \times (T-t_0+1)} \to \R$ is a forward separable function with associated representation $\{ \phi_i\}_{i=t_0}^T$; $f$ : $\mathbb{R}^n \times \mathbb{R}^m \times \mathbb{N} \times \R^q \to \mathbb{R}^n$; $\psi_{f, t_0}$ {is measurable} and $\Phi_{f, t_0}$ are the state and trajectory map respectively defined in Definition \ref{defn: the trajectory map}; $U_i$ is assumed to be some compact subset of $\R^{m \times (i - t_0)}$; $X_i \subset \R^{n \times (i -t_0 +1) }$; $[\mbf v]_{t_0}^{T-1} = [v(t_0),..,v(t-1)] \in \R^{q \times (T-t_0 )}$; $\mathbb E_{\mbf v}$ is the expectation with respect to the random variable $\mbf v$. Define $J_{t_0}^{\mcl H_s*}= \mbb E_{[\mbf v]_{t_0}^{T-1}} \left(J^{\mcl H_s}_{t_0}( \Phi_{f,t_0}(\pi^{H_s*}, x_0, T, [\mbf v]_{t_0}^{T-1}) )\right)$ as the expected cost of using an optimal policy when applied to $\mcl H_s(t_0,x_0)$.

\textbf{Change in Variables for Stochastic Problems:} Unlike in the deterministic case the solution to stochastic DP problems, such as \eqref{eqn:Forward sep DP stochastic}, is now a policy $\pi \in \Pi$ and not a definite input and state sequence $\mathbf{u}^*\in \R^{m \times (T-t_0)}$ and $\mathbf{x}^*\in \R^{n \times (T-t_0+1)}$, such as in \eqref{eqn:opt}. This is because the optimal sequence of inputs, $\mathbf u^*$, that results in an optimal trajectory, $\mathbf x^*$, will depend on the instantiation of the random variables. This change of notation demonstrates that the solution to DP problems involving stochastic dynamics no longer belongs to some finite dimensional space, $(\mathbf{u}^*,\mathbf{x}^*) \in \R^{m \times (T-t_0)}\times \R^{n \times (T-t_0+1)}$, but rather an infinite dimensional functional space $\pi^* \in \Pi$.

\subsection{Stochastic Additively Separable DP Problems}
{In the special case when the objective function of \eqref{eqn:Forward sep DP stochastic} is an additively separable function, as per Definition \ref{defn: additively seperbale function}, given as
	\begin{align} \label{eqn:DP stochastic}
	J_{t_0}^{\mcl Q}(\mbf u, \mbf x)= \sum_{t=t_0}^{T-1} c_t(x(t),u(t)) + c_T(x(T)),
	\end{align}
	 we denote the DP problem \eqref{eqn:Forward sep DP stochastic} by $\mcl Q(t_0,x_0)$. }

{stochastic DP problems of form $\mcl Q(t_0,x_0)$ are often referred to as Markov Decision Processes (MDP), and are sometimes denoted by the tuple $\{\{X_t\}_{t_0}^{T},\{U_t\}_{t_0}^{T-1},\psi,\{Q_t\}_{t_0}^{T-1},\{c\}_{t_0}^T\}$; where $\psi(x,v,t)=\{u \in U_t: f(x,u,t,v) \in X_{t+1} \}$, $Q_t(B|x,u)= \int_{B} \mathds{1}_{B}(f(x,u,t,v)) \phi(v) dv $, and $\phi(v)$ is the probability density function of the random variable $v$.
}

\subsection{The Principle Of Optimality For Stochastic Problems}
As discussed in \cite{Sniedovich_2002_eureka} the extension of the Principle of Optimality to the stochastic case is non-trivial. We first give an example from \cite{Porteus_1975} of a stochastic DP problem which shows that an optimal policy may not be optimal for every instantiation of the random variables at future time steps.


Let us consider the following stochastic DP problem $\mcl W(0,x_0)$,
\begin{align} \label{opt: W}
&\pi^*= \arg \min_{\pi \in \Pi} \mbb E_{v(0)} \left( J_{0}( \Phi_{f,0}(\pi, x_0,1, [ v(0)]))    \right)\\ \nonumber
& \text{subject to: } v(0) \sim U[0,1] , \quad x(0)=x_0.
\end{align}
Here $J_{t}(\mbf u, \mbf x)= -\sum_{n=t}^{1} u(n)$, $f(x,u,t,v) = v$, and $\pi \in \Pi \iff \pi(x,t)\in \{0,1\} \forall x \in \R, t=0,1$.


\begin{cex} \label{cex: principle of optimality is not valid for stochastic}
	The policy $\pi(x,t)= \begin{cases} 1 \text{ if } x\in[0,1)\\ 0 \text{ if } x=1 \end{cases}$ is optimal for the problem $\mcl W(0,0)$ \eqref{opt: W} but not optimal for the problem $\mcl W(1,1)$.
\end{cex}
\begin{proof}
	Clearly $J_{0}(\mbf u, \mbf x) \ge -2$ for all $(\mbf u, \mbf x) \in \{0,1\}^{2} \times \R^3$ and $J_{0}(\mbf u, \mbf x)= -2$ is attainable using the input $(u(0),u(1))=(1,1)$; therefore any solution of $\mcl W(0,0)$ will minimize the objective function to a value of -2. Now using the law of total expectation we get,
	\begin{align*}
	& \mbb E_{v(0)} \left( J_{0}( \Psi_{f,0}(\pi, 0,1, [ v(0),v(1)]))   \right) \\
	& = -\mbb E_{v(0)} \left( \pi(0,0) + \pi(v(0),1) \right)\\
	& = -\pi(0,0) - \mbb E_{v(0)} \left( \pi(v(0),1) | v(0) \in [0,1) \right) \mbb P_{v(0)}(v(0) \in [0,1) )\\
	& \qquad - \mbb E_{v(0)} \left( \pi(v(0),1) | v(0) =0 \right) \mbb P_{v(0)}(v(0) =0 ) \\
	& = -2,
	\end{align*}
	since the probability of a continuous random variable (such as a uniformly distributed random variable) taking a particular value is 0. Thus it follows the policy $\pi$ is optimal for $\mcl W(0,0)$. Trivially $\pi$ is not optimal for $\mcl W(1,1)$ as the value of the objective functions becomes 0 under $\pi$ whereas the input $u(1)=1$ produces a smaller objective function value of -1.
\end{proof}

Clearly, for the stochastic DP problems of form $\mcl H_s(t_0,x_0)$ \eqref{eqn:Forward sep DP stochastic}, such as $\mcl W(0,0)$ \eqref{opt: W}, any optimal policy $\pi^*$ does not always result in the same trajectory $\mathbf x=(x(t_0),..., x(T))$; as this is dependent on the instantiations of the underlying random variables, $[\mbf v]_{t_0}^{T-1}$. As Counterexample \ref{cex: principle of optimality is not valid for stochastic} has shown there {exist} stochastic DP problems, with additively separable objective functions, that have optimal policies that are no longer optimal for future timesteps if certain instantiations of the underlying random variables are realized. Therefore, it is too restrictive to extend Definition \ref{defn: principle of optimality deterministic}, the Principle of Optimality for the deterministic case, to the stochastic case by requiring stochastic DP problems satisfying the Principle of Optimality to be such that their optimal policies are also optimal for each instantiation at any future time step. With this in mind and motivated by the work of \cite{Hinderer_1970} we now give a probabilistic definition of the Principle of Optimality for stochastic DP problems.
	

\begin{defn} \label{defn: stochastic principle of optimality}
	For a stochastic DP problem $\mcl H_s(t_0,x_0)$ \eqref{eqn:Forward sep DP stochastic} with an optimal policy $\pi^* \in \Pi$ and associated state map $\psi_{f,t_0}$, defined in definition \ref{defn: the trajectory map}, let us denote the set indexed by $k \ge t_0$,
\begin{align*}
Y_k= & \{[\mbf v]_{t_0}^{k-1} \in \R^{q \times (k -t_0)}: \\
& \pi^* \text{ does not solve } \mcl H_s(k,\psi_{f,t_0}(\pi^*, x_0,k, [\mbf v]_{t_0}^{k-1})) \}
\end{align*}
	where $[\mbf v]_{t_0}^{k-1} =[v(t_0),...,v(k-1)] \in \R^{q \times (k-t_0)}$. We say stochastic DP problems of the form $\mcl H_s(t_0,x_0)$ \eqref{eqn:Forward sep DP stochastic} \textbf{satisfy the Principle of Optimality} if for any $k\ge t_0$ we have
	\[
	\mathbb P_{ [\mbf v]_{t_0}^{k-1}}( [\mbf v]_{t_0}^{k-1} \in Y_k )=0.
	\]
	Here $\mbb P_{ [\mbf v]_{t_0}^{k-1} }$ is the probability measure associated with the random variable $[\mbf v]_{t_0}^{k-1} \in \R^{q \times (k-t_0)}$, $v(t) \sim \mcl N(\mbf 0, I_{q \times q} )$ for $t \in \{t_0,...,k-1\}$.
	
\end{defn}

{We next show that stochastic DP problems with additively separable objective functions (MDP's) satisfy the Principle of Optimality as formulated in Definition \ref{defn: stochastic principle of optimality}. }

\begin{lem} \label{lem: addative stochastic functions satisfy the principle of optimality}
	A stochastic DP problem of Form $\mcl Q(t_0,x_0)$ \eqref{eqn:DP stochastic} satisfies the Principle of Optimality as formulated in Definition \ref{defn: stochastic principle of optimality}.
\end{lem}

\begin{proof}
	
	Suppose $\pi^*$ solves $\mcl Q(t_0,x_0)$. For $k>t_0$ and the state map $\psi_{f,t_0}$ associated with $\mcl Q(t_0,x_0)$ let us recall the set defined in Definition \ref{defn: stochastic principle of optimality},
	\small{ \begin{align*}
	Y_{k}: =  \{[\mbf v]_{t_0}^{k-1} & \in \R^{q \times (k -t_0)}: \pi^* \text{ does not solve } \mcl Q(k,x_{\pi^*}( [\mbf v]_{t_0}^{k-1})) \}.
	\end{align*} } \normalsize
	Here $[\mbf v]_{t_0}^{k-1}: =[v(t_0),...,v(k-1)] \in \R^{q \times (k-t_0)}$, and we use the short-hand $x_{\pi^*}([\mbf v]_{t_0}^{k-1}):=\psi_{f,t_0}(\pi^*, x_0,k, [\mbf v]_{t_0}^{k-1} )$.
	
	Now for contradiction suppose there exists $k \in \{t_0,...,T\}$ such that $\mathbb P_{ [\mbf v]_{t_0}^{k-1}}( [\mbf v]_{t_0}^{k-1} \in Y_{k} )>0$; where $v(t) \sim \mcl N(\mbf 0, I_{q\times q})$ for $t \in \{t_0,...,k-1\}$. For $[\mbf v]_{t_0}^{k-1} \in Y_{k}$  we know the policy $\pi^*$ is not optimal for $\mcl Q(k,x_{\pi^*}([\mbf v]_{t_0}^{k-1}) )$ and thus there exists a feasible policy $\theta \in \Pi$ such that,
		{ \begin{align} \nonumber
	 & \mbb E_{  [\mbf v ]_{t_0}^{T-1}} \bigg(  J_{k }^{\mcl Q}(\Phi_{f,k}  (\theta, x_{\pi^*}([\mbf v]_{t_0}^{k-1}),T, [\mbf v]_{k}^{T-1} )) \bigg|  [\mbf v]_{t_0}^{k-1}  \in Y_{k}   \bigg) \\ \label{ineq: pi on is not optimal}
	& <\\ \nonumber
	&  \mbb E_{[\mbf v ]_{t_0}^{T-1}} \bigg( J_{k}^{\mcl Q}(\Phi_{f,k}  (\pi^*, x_{\pi^*}([\mbf v]_{t_0}^{k-1}) ,T, [\mbf v]_{k}^{T-1} )) \bigg| [\mbf v]_{t_0}^{k-1} \in Y_{k}  \bigg).
	\end{align} }
\normalsize
	Now let us consider the map,
	\small{ \begin{align} \label{eta}
	\hat{\pi}_t([x(t_0),...,x(t)])= \begin{cases}
		\theta(x(t),t) \text{ if } t\ge k, x(k) \in \psi_{f,t_0}(\pi^*,x_0,k,Y_{k}) \\
		\pi^*(x(t),t) \text{ otherwise }.
		\end{cases}
		\end{align}} \normalsize
Using Lemma \ref{lem: Optimal policy is Markov}, there exists a policy $\alpha \in \Pi$ such that \eqref{identity: markov produces same objective function value} holds for $\{\hat{\pi}_t\}$ defined in \eqref{eta}. We will now show the policy $\alpha$ contradicts that $\pi^*$ be an optimal policy for $\mcl Q(t_0,x_0)$. We first note using \eqref{identity: markov produces same objective function value} and the law of total probabilities,
	{ 	\begin{align} \label{eqn: law of total prob}
	 & \mbb E_{[\mbf v]_{t_0}^{T-1} } \left( J_{t_0}^{\mcl Q}( \Phi_{f,t_0}(\alpha, x_0,T, [\mbf v]_{t_0}^{T-1} ))   \right)\\ \nonumber
	 &  = \mbb E_{[\mbf v]_{t_0}^{T-1}}\left( J_{t_0}^{\mcl Q}(\Phi_{f,t_0}(\hat{\pi}, x_0,T, [\mbf v]_{t_0}^{T-1} )) \right)   \\ \nonumber
	 & = \mbb E_{[\mbf v]_{t_0}^{T-1} } \left( J_{t_0}^{\mcl Q}( \Phi_{f,t_0}(\hat{\pi}, x_0,T, [\mbf v]_{t_0}^{T-1} ) | [\mbf v]_{t_0}^{k-1} \in Y_{k}   \right)\\ \nonumber
	 & \qquad \qquad  \mathbb P_{ [\mbf v]_{t_0}^{k-1}}( [\mbf v]_{t_0}^{k-1} \in Y_{k} ) \\ \nonumber
	 & \qquad + \mbb E_{[\mbf v]_{t_0}^{T-1} } \left( J_{t_0}^{\mcl Q}( \Phi_{f,t_0}(\hat{\pi}, x_0,T, [\mbf v]_{t_0}^{T-1} ) ) | [\mbf v]_{t_0}^{k-1} \notin Y_{k}   \right) \\ \nonumber
	 & \qquad \qquad \mathbb P_{ [\mbf v]_{t_0}^{k-1}}( [\mbf v]_{t_0}^{k-1} \notin Y_{k} ).
		\end{align} } 	\normalsize
	We recall the additive structure of $J_{t_0}^{\mcl Q}$
	\begin{align*}
	J_{t_0}^{\mcl Q}(\mbf u, \mbf x)= \sum_{t=t_0}^{T-1} c_t(x(t),u(t)) + c_T(x(T)),
	\end{align*}
	where $\mbf u=(u(t_0),...,u(T-1)) \text{ and } \mbf x =(x(t_0),...,x(T))$ and $c_T(x):\mathbb{R}^n \to \mbb R$, $c_t(x,u):\mbb R^n \times \R^m \to \R $ for $t=t_0,\cdots T-1$.

Now using the fact $\hat{\pi}_t([x(t_0),...,x(t)])= \pi^*(x(t),t)$ for all $t<k$, $\hat{\pi}_t([x(t_0),...,x(t)])= \theta(x(t),t)$ if $t\ge k$ and $x(k) \in \psi_{f,t_0}(\pi^*,x_0,k,Y_{k})$, linearity of the expectation and the inequality \eqref{ineq: pi on is not optimal} we have,
		\small{
		\begin{align} \label{above inequality}
		& \mbb E_{[\mbf v]_{t_0}^{T-1} } \left( J_{t_0}^{\mcl Q}( \Phi_{f,k}  (\pi^*, x_0 ,T, [\mbf v]_{t_0}^{T-1} ) ) \bigg| [\mbf v]_{t_0}^{k-1} \in Y_{k}   \right) \\ \nonumber
		& = \mbb E_{[\mbf v]_{t_0}^{T-1}}\left( \sum_{t=t_0}^{k-1}{c_{t}(x_{\pi^*}([\mbf v]_{t_0}^{t-1}),\pi^*(x_{\pi^*}([\mbf v]_{t_0}^{t-1}),t))} \bigg| [\mbf v]_{t_0}^{k-1} \in Y_{k}    \right) \\
		\nonumber
		&  + \mbb E_{[\mbf v ]_{t_0}^{T-1}} \bigg(  J_{k}^{\mcl Q}(\Phi_{f,k}  (\theta, x_{\pi^*}([\mbf v]_{t_0}^{k-1}),T, [\mbf v]_{k}^{T-1} )) \bigg|  [\mbf v]_{t_0}^{k-1}  \in Y_{k}   \bigg) \\ \nonumber
		& < \mbb E_{[\mbf v]_{t_0}^{T-1}}\left( \sum_{t=t_0}^{k-1}{c_{t}(x_{\pi^*}([\mbf v]_{t_0}^{t-1}),\pi^*(x_{\pi^*}([\mbf v]_{t_0}^{t-1}),t))} \bigg| [\mbf v]_{t_0}^{k-1} \in Y_{k}    \right) \\
\nonumber
&  + \mbb E_{[\mbf v ]_{t_0}^{T-1}} \bigg(  J_{k }^{\mcl Q}(\Phi_{f,k}  (\pi^*, x_{\pi^*}([\mbf v]_{t_0}^{k-1}),T, [\mbf v]_{k}^{T-1} )) \bigg|  [\mbf v]_{t_0}^{k-1}  \in Y_{k}   \bigg) \\ \nonumber
		&= 	\mbb E_{[\mbf v]_{t_0}^{T-1}}\left( J_{t_0}^{\mcl Q}(\Phi_{f,t_0}(\pi^*, x_0,T, [\mathbf v]_{t_0}^T))  \bigg| [\mbf v]_{t_0}^{k-1} \in Y_{k}   \right).
		\end{align} } \normalsize
		Therefore using \eqref{eqn: law of total prob}; the fact $\hat{\pi}_t([x(t_0),...,x(t)])= \pi^*(x(t),t)$ if $x(k) \notin \psi_{f,t_0}(\pi^*,x_0,k,Y_{k, \pi^*})$; the total law of probability; the above inequality \eqref{above inequality}; and the assumption $\mathbb P_{ [\mbf v]_{t_0}^{k-1}}( [\mbf v]_{t_0}^{k-1} \in Y_{k} )>0$ (so the inequality remains strict) we derive,
				\begin{align}
		& \mbb E_{[\mbf v]_{t_0}^{T-1} } \left( J_{t_0}^{\mcl Q}( \Phi_{f,t_0}(\alpha, x_0,T, [\mbf v]_{t_0}^{T-1} ))   \right)   \\ \nonumber
		& \qquad < \mbb E_{[\mbf v]_{t_0}^{T-1} }\left( J_{t_0}^{\mcl Q}(\Phi_{f,t_0}(\pi^*, x_0,T, [\mathbf v]_{t_0}^{T-1}))    \right).
		\end{align}
		This contradicts the fact $\pi^*$ is an optimal policy for $\mcl Q(t,x)$. Therefore we conclude $P_{ [\mbf v]_{t_0}^{k-1}}( [\mbf v]_{t_0}^{k-1} \in Y_{k} )=0$ showing DP problems of the form $\mcl Q(t_0,x_0)$ satisfy Definition \ref{defn: stochastic principle of optimality} and hence satisfy the Principle of Optimality. \end{proof}

We will now state Bellman's equation for stochastic DP problems of the form $\mcl Q(t,x)$.


\begin{prop}[\cite{Lerma_1996}]
	 	For stochastic DP problems of the form $\mcl Q(t,x)$ in \eqref{eqn:DP stochastic} with optimal objective values $J_{t}^{\mcl Q*}$, define the function $F(x,t)= J_{t}^{\mcl Q*}$. Then the following hold for all $x \in X_t$,	 	
	\begin{align} \label{Bellman_stochastic}
	& F(x,t)=  \inf_{u}\{ c_t(x,u )  + \mathbb{E}_{v}[F(f(x,u,t,v),t+1) ]\}. \\ \nonumber
	& F(x,T)=c_T(x). \nonumber
	\end{align}  \normalsize
\end{prop}

We see in the next corollary that if we are able to solve the stochastic Bellman equation \eqref{Bellman_stochastic} then we are able to construct an optimal policy that solves \eqref{eqn:DP stochastic}.

\begin{cor}[\cite{Lerma_1996}] \label{cor: optimal policy of stochatsic}
	Consider a stochastic DP problem of the form $\mcl Q(t_0,x_0)$ in \eqref{eqn:DP stochastic}. Suppose $F(x,t)$ satisfies Equation~\eqref{Bellman_stochastic} and suppose there exists a policy such that,
	\[
	\theta(x,t) \in \arg \min_{u \in \Gamma_{t,x}}\{c_t(x,u)+ \mathbb{E}_{ v}[F(f(x,u,t,v),t+1)]\}.
	\]
	Then the policy $\theta$ {solves }$\mcl Q(t_0,x_0)$.
\end{cor}

 \subsection{{State Augmentation For Stochastic DP problems}}
 Analogous to the deterministic case shown in Lemma \ref{lem: augmented optimization is equivalent to original}, for a stochastic DP problem of the form $\mcl H_s(t_0,x_0)$ \eqref{eqn:Forward sep DP stochastic} we can use the separable representation maps $\{ \phi_i\}_{i=t_0}^T$ of the objective function $J^{\mcl H_s}_{t_0}$ to construct an equivalent DP problem of form $\mcl Q(t_0,x_0)$ \eqref{eqn:DP stochastic} {by using state augmentation}. 
\section{{Numerically Solving Additively Separable Stochastic DP Problems}} \label{sec: numerically solving stochastic problems}
{In Section~\ref{sec:SDP}, we proposed a state-augmentation scheme for converting a stochastic forward separable DP problem, of form $\mcl H_s(t_0,x_0)$ \eqref{eqn:Forward sep DP stochastic}, to an equivalent additively separable DP problem, of form $\mcl Q(t_0,x_0)$~\eqref{eqn:DP stochastic}. In this section we propose a scheme to numerically solve stochastic additively separable DP problems of the form $\mcl Q(t_0,x_0)$ \eqref{eqn:DP stochastic}.}


{To numerically solve problems of the form $\mcl Q(t_0,x_0)$, we propose a discretization scheme similar to the one detailed in Section \ref{sec:numerical}. However, unlike in the deterministic case, the presence of random variables, $v \sim \mcl N(0,1)$, implies a non-compact state space. This requires the use of state projection onto appropriately constructed approximating compact sets. }

\subsection{Constructing An Approximated DP Problem With Compact State Space} \label{sec: approximately solving Q}
Consider the stochastic DP problem $\mcl Q(t_0,x_0)$ with compact control space $U=[\underline{u},\bar{u}]^m$ and underlying random variables $v \sim \mcl N( \mbf 0, I_{q \times q})$. As in \cite{Dufour_2012}, we assume $\forall \epsilon>0$ that there exists a compact set $H_{\eps,t}=[\underline{x}_{\eps,t},\bar{x}_{\eps,t}]^n \subset X$ (that depends on $\eps$ and $t$) such that $x_0 \in H_{\eps,0}$ and,
\begin{equation}
\sup_{x \in H_{\eps,t}, u \in U} \mathbb{P}_{ v}(f(x,u,t,v) \notin H_{\eps,t+1}) < \epsilon.
\end{equation}
We then construct the associated compact stochastic DP approximation to $\mcl Q(t_0,x_0)$ denoted by $\mcl Q_{\eps,k}(t_0,x_0)$,
\small{ \begin{align}
	&\arg \min_{\pi \in \Pi} \;\; \mbb E_{ [\mbf v]_{t_0}^{T-1} }\left( J_{t_0}^{\mcl Q}(\Phi_{\tilde{f},t_0}(\pi, x_0,T, [\mbf v]_{t_0}^{T-1}))  \right) \label{eqn:DP stochastic_approx}\\ \nonumber
	&\text{subject to:  }  \text{ } \psi_{\tilde{f},t_0}(\pi, x_0, t, [\mbf v]_{t_0}^{t-1}) \in \tilde{X}_{\eps,t, k} \text{ for  } t=t_0,..,T,  \\ \nonumber
	&  \pi(x,t) \in \tilde{U}_k \text{ and }   v(t) \in \R^q \sim \mathcal{N}(\mbf 0,I_{q \times q}) \forall x \in X_t, \forall t=t_0,..,T-1,
	\end{align} } \normalsize
where $\tilde{f}(x,u,t,v)=\arg \min_{y \in X_{\eps,t+1, k}}\{||y-f(x,u,t,v)||_2\}$, $\tilde{X}_{\eps,t, k}=\{x_{1,t},...,x_{k,t}\}^n$ such that $\underline{x}_{\eps,t}=x_{1,t}<x_{2,t}<...<x_{k,t}=\bar{x}_{\eps,t}$ and $||x_{i+1,t}-x_{i,t}||_2=\frac{\bar{x}_{\eps,t}-\underline{x}_{\eps,t}}{k}$ for $1 \le i \le k-1$, $\tilde{U}_k=\{u_1,...,u_k\}^m$ such that $\underline{u}=u_1<u_2<...<u_k=\bar{u}$ and $||u_{i+1}-u_{i}||_2=\frac{\bar{u}-\underline{u}}{k}$ for $1 \le i \le k-1$, and $[\mbf v]_{t_0}^{T-1} = [v(t_0),..,v(T-1)] \in \R^{q \times (T-t_0)}$.

Analogous to the deterministic case, an optimal policy $\pi^*_{\eps,k}$ for $\mcl Q_{\eps,k}(t_0,x_0)$ can be found exactly by iteratively solving Bellman's equation \eqref{Bellman_stochastic}. One can then construct a feasible policy for $\mcl Q(t_0,x_0)$ using,
\begin{equation} \label{eqn:constructed policy_stochastic}
\theta_{\eps,k}(x,t)= \arg \min_{u \in \Gamma_{t,x} } ||\pi^*_{\eps,k}( \arg \min_{y\in {X_{\eps,t,k}}}\{||y-x||_2\},t) - u ||_2\in \Pi
\end{equation}
where $\Gamma_{t,x}$ is the set of feasible controls at time $t \in \{0,...,T-1\}$ and state position $x \in \R^n$ for $\mcl Q(t_0,x_0)$ \eqref{eqn:DP stochastic} and $X_{\eps,t,k}$ is the state grid constraint in the problem $\mcl Q_{\eps,k}(t_0,x_0)$ \eqref{eqn:DP stochastic_approx}.

If $\mcl Q(t_0,x_0)$ satisfies assumption (A1) to (A4) {from Theorem 3.5} \cite{Dufour_2012} then \small{\begin{equation} \label{eqn:convergance conjecture_stochatsic} \lim_{\eps \to 0, k \to \infty} \left| \mbb E_{ [\mbf v]_{t_0}^{T-1} }\left( J_{t_0}^{\mcl Q}(\Phi_{\tilde{f},t_0}(\theta_{\eps,k}, x_0,T, [\mbf v]_{t_0}^{T-1}))  \right) - J_{t_0}^{\mcl Q*} \right|=0,\end{equation} } \normalsize
 where $J_{t_0}^{\mcl Q*}= \mbb E_{[\mbf v]_{t_0}^{T-1}} \left(J^{\mcl Q}_{t_0}( \Phi_{f,t_0}(\pi^{\mcl Q*}, x_0, T, [\mbf v]_{t_0}^{T-1}))\right)$ is the expected cost of using an optimal policy when applied to $\mcl Q(t_0,x_0)$.


\begin{figure*}
	\centering
	\includegraphics[width=18cm,height=8cm]{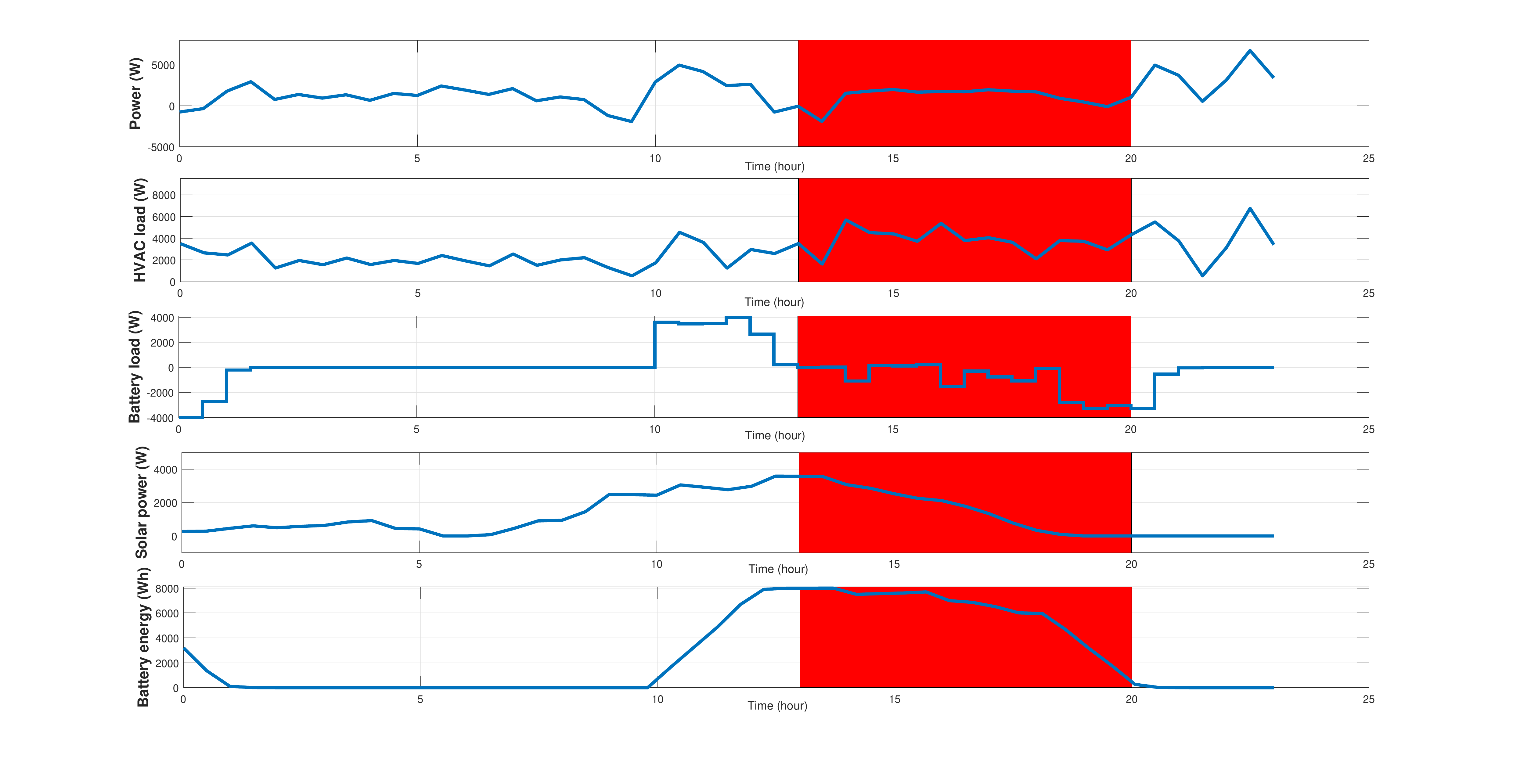}
	\vspace{-20pt}
	\caption{The trajectory the algorithm produces for randomly generated stochastic solar data. The supremun of the power is 1.05788(kw) and the cost is \$47.7211.}
	\label{fig:stochastic}
	\vspace{-20pt}
\end{figure*}

\section{{Summary: Solving NFS DP Problems Using Augmentation And Discretization}} \label{sec: augmentation and discretization methods}
\begin{figure}
	\includegraphics[scale=0.285]{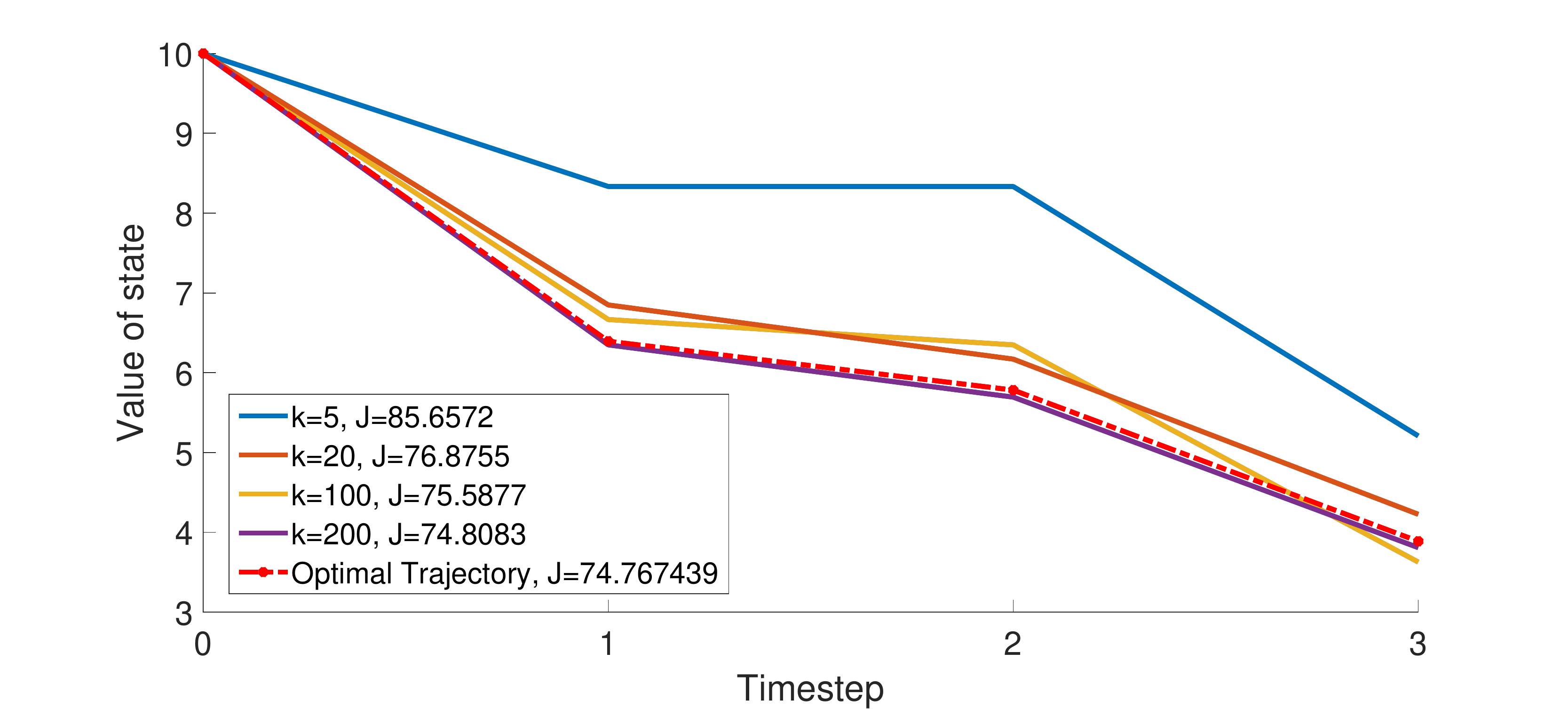}
	\vspace{-25pt}
	\caption{The resulting state trajectories from using the policy constructed from $P_k(t_0,x_0)$ in the DP Problem \eqref{eqn:Duan example}.}
	\label{fig:Duan example}
	\vspace{-20pt}
\end{figure}
{Given a DP problem with a NFS objective function, with known representation maps, we have shown in Section \ref{sec:ADP} and Section \ref{sec:SDP} how to construct equivalent DP problems with additively separable objective functions. We have furthermore proposed discretization schemes in Section \ref{sec:numerical}, for the deterministic case, and Section \ref{sec: approximately solving Q}, for the stochastic case, to solve DP problems with additively separable objective functions. We now summarize these results by proposing the following steps for solving a non-separable DP problem. Given a DP problem of the from $\mcl H(t_0,x_0)$ \eqref{opt:forward_sep}, or $\mcl H_s(t_0,x_0)$ \eqref{eqn:Forward sep DP stochastic} if stochastic, we do the following:}
\begin{enumerate}
	\item { Find a NFS representation of the objective function (Eqn.~\eqref{forward_sep_def}) with associated representation maps. One approach to this is to use Section \ref{subsec: algebra of NFSF's} which details how to combine known NFS functions, with known representation maps, in order to find potential representation maps for other NFS functions.}
	\item {Construct the associated augmented DP problem of form $\mcl P(t_0,x_0)$ \eqref{eqn:DP}, if deterministic, or $\mcl Q(t_0,x_0)$ \eqref{eqn:DP stochastic}, if stochastic.}
	\item {Use discretization to approximate the augmented DP problem using Form $\mcl P_k(t_0,x_0)$ \eqref{eqn: approx opt}, if deterministic, or $\mcl Q_{\eps, k}(t_0,x_0)$ \eqref{eqn:DP stochastic_approx}, if stochastic. }
	\item { Numerically solve $\mcl P_k(t_0,x_0)$ or $\mcl Q_{\eps, k}(t_0,x_0)$ for a sufficiently large $k \in \N$.}
	\item { Construct a feasible policy for the original DP problem from an optimal policy of $\mcl P_k(t_0,x_0)$ or $\mcl Q_{\eps, k}(t_0,x_0)$ using \eqref{eqn:constructed policy} or \eqref{eqn:constructed policy_stochastic}. }
\end{enumerate}


{To illustrate how we use state augmentation and discretization methods we consider the following DP problem from \cite{Duan}.}
\begin{align} \label{eqn:Duan example}
& \min J =  x(3)^2[u(0)^2+u(1)^2+u(1)u(2)^2]^{\frac{1}{2}}\\ \nonumber
& \hspace{2.5cm} + [u(0)^2+u(1)^2+u(1)u(2)^2]^2\\ \nonumber
& \text{subject to,} \quad x(t+1)=\frac{x(t)}{u(t)} \quad \text{for } t \in \{0,1,2\}\\ \nonumber
& x(0)=10, \quad u(0),u(1),u(2) \ge 0.
\end{align}
In \cite{Duan} an analytic solution for \eqref{eqn:Duan example} was found to be:
\begin{align*}
\mathbf x^* = \bmat{10 \\ 6.3943938 \\ 5.782475 \\ 3.8882658}, \hspace{0.1cm} \mathbf u^*=  \bmat{1.5638699 \\ 1.105823 \\ 1.4871604}, \hspace{0.1cm} J^*=74.767439.
\end{align*}
The objective function $J$ in \eqref{eqn:Duan example} is NFS and has a representation dimension of 2. This can be shown by writing $J$ {in} the form of \eqref{forward_sep_def} using the functions,
\begin{align*}
& \phi_0(x,u)=u^2, \quad \phi_1 \left(x,u,w \right)=\bmat{w+u^2\\ u} \\
& \qquad \phi_2 \left(x,u,\bmat{w_1 \\ w_2} \right)=w_1+w_2^2 u^2, \\
& \qquad \phi_3 \left(x, w \right)=x^2 \sqrt{w} + w^2.
\end{align*}
The DP Problem \eqref{eqn:Duan example} can now be written {in} the form of $\mcl A(t_0,x_0)$ using state augmentation, as
\begin{align} \label{eqn:duan augment}
& \min\{ z_1(3)^2 \sqrt{z_3(3)} + z_3(3)^2 \}\\ \nonumber
&\text{subject to,}\\ \nonumber
& z_1(t+1)=\frac{z_1(t)}{u(t)}, \hspace{0.1cm} z_2(t+1)= \begin{cases}
u(t) \text{ if t=1}\\
0 \text{ otherwise}
\end{cases} \forall t \in \{0,1,2\},\\ \nonumber
& z_3(1)= u(1)^2, \quad z_3(2)=z_3(1) + u(1)^2, \\ \nonumber
& z_3(3)=z_3(2)+z_2(2)^2 u(2),  \\ \nonumber
& z_1(0)=10, \hspace{0.1cm} z_2(0)=0, \hspace{0.1cm} z_3(0)=0 \quad u(0),u(1),u(2) \ge 0.
\end{align}
The DP Problem \eqref{eqn:duan augment} is now a special case of $\mcl P(t_0,x_0)$ and equivalent to the original DP Problem \eqref{eqn:Duan example}. The associated approximated DP problem of the form $\mcl P_k(t_0,x_0)$ \eqref{eqn: approx opt} can now be found by selecting appropriate compact state and control spaces; $X \subset \R^3$ and $U \subset \R$. A feasible policy for \eqref{eqn:Duan example} is then constructed from an optimal policy of the associated $\mcl P_k(t_0,x_0)$ using \eqref{eqn:constructed policy}. Figure \ref{fig:Duan example} shows the state trajectories by following different constructed policies for various values of $k$. It is seen that for $k=200$ the algorithm produces a solution within three significant figures of the analytic optimal objective function for \eqref{eqn:Duan example}.

\section{Application To The Energy Storage Problem}\label{sec:battery1}
\label{sec:ProblemStatement}
We apply our augmented DP methodology to the scheduling of batteries in the presence of demand charges {and show that our proposed algorithm outperforms existing heuristics, such as \cite{multi-objective} (approximately \$0.98 savings)}. To do this, we propose a simple model for the dynamics of battery storage. We then formulate the objective function using electricity pricing plans which include demand charges. We see that the system described becomes a DP problem of the form $\mcl S(t_0,x_0)$ \eqref{eqn:S}; which can be tractably solved as it has a NFS objective function. {We will first solve the battery scheduling problem in the deterministic case based on real {historical} solar data. Later, we develop a stochastic Markov model that generates similar solar data to that seen in Tempe, Az.   }
\subsection{Battery Dynamics} \label{sec:battery dynamics}
We model the energy stored in the battery using the difference equation:
\begin{equation}
e({k+1})=\alpha (e({k})+\eta u({k})\Delta t),
\end{equation}
where $e(k)$ denotes the energy stored in the battery at time step $k$, $\alpha$ is the bleed rate of the battery, $\eta$ is the efficiency of the battery, $u({k})$ denotes the charging/discharging $(+/-)$ at time step $k$ and $\Delta t$ is the amount of time passed between each time step. Moreover we denote the maximum charge and discharge rate by $\bar{u}$ and $\underline{u}$ respectively. Thus we have the constraint that $u({k}) \in [\underline{u},\bar{u}]:=U $ for all $k$. Similarly we also add the constraint $e(k) \in [\underline{e},\bar{e}]:=X$ for all $k$ where $\underline{e}$ and $\bar{e}$ are the capacity constraints of the battery (typically $\underline{e}=0$).
\subsection{The Objective Function}
Let us denote $q(k)$ as the power supplied by the grid at time step k.
\begin{equation}
q(k)=q_{a}(k)-q_{s}(k)+u({k}),
\end{equation}
where $q_{a}(k)$ and $q_{s}(k)$ are the power consumed by HVAC/appliances and the power supplied by solar photovoltaics at time step $k$ respectively. For now, it is assumed that both are known a priori.


To define the cost of electricity we divide the day $t \in [0,T]$ into on-peak and off-peak periods. We define an off peak period starting from 12am till $t_{\text{on}}$ and $t_{\text{off}}$ till 12am. We define an on-peak period between $t_{\text{on}}$ till $t_{\text{off}}$. The Time-of-Use (TOU, \$ per kWh) electricity cost during on-peak and off-peak is denoted by $p_{\text{on}}$ and $p_{\text{off}}$ respectively. We further simplify this as $p_k = p_{on}$ if $ k \in T_{on}$ and $p_k = p_{off}$ if $ k \in T_{off}$ where $T_{on}$ and $T_{off}$ are the on-peak and off-peak hours, respectively. These TOU charges define the first part of the objective function as:
\begin{align*}
J_{\text{E}}(\mbf u, \mbf e)&=p_{\text{off}} \sum_{k=0}^{t_{\text{on}}-1}{q(k)\Delta t} + p_{\text{on}} \sum_{k=t_{\text{on}}}^{t_{\text{off}}-1}{q(k)\Delta t} + p_{\text{off}} \sum_{k=t_{\text{off}}}^{T}{q(k)\Delta t}\\
&=\sum_{k \in [0,T]} p_k (q_{a}(k)-q_{s}(k))\Delta t+\sum_{k \in [0,T]} p_k u({k})\Delta t
\end{align*}
where the daily terminal timestep is $T=24/\Delta t$. 

We also include a demand charge, which is a cost proportional to the maximum rate of power taken from the grid during on-peak times. This cost is determined by $p_{d}$ which is the price in \$ per kW. Thus it follows the demand charge will be:
\begin{align*}
J_{D}(\mbf u, \mbf e)
& = p_{d} \max_{k \in \{t_{\text{on}},....,t_{\text{off}}-1\}}\{q_{a}(k)-q_{s}(k)+u({k})\}.
\end{align*}

\subsection{24 hr Optimal Residential Battery Storage Problem}
We may now define the problem of optimal battery scheduling in the presence of demand and Time-of-Use charges, denoted $\mcl D(0,e_0)$.
\begin{align} \label{eqn: Determinist battery problem}
& \min_{\mbf u, \mbf e}\{J_{E}(\mbf u, \mbf e)+J_{D}(\mbf u, \mbf e)\} \text{     subject to}\\ \nonumber
& e(k+1)=\alpha (e(k)+\eta u({k})\Delta t)\text { for } k=0,...,T-1\\ \nonumber
& e_{0}=e0 \text{ }, e(k)\in X, \text{ } u(k) \in U \text { for } k=0,...,T, \\ \nonumber
& \mbf u=(u(0),...,u(T-1)) \text{ and } \mbf e =(e(0),...,e(T))
\end{align}
where recall $U:=[\underline{u},\bar{u}]$ and $X:=[\underline{e},\bar{e}]$.
\begin{prop}
	Problem $\mcl D(0,e_0)$ is a special case of $\mcl S(t_0,x_0)$ \eqref{eqn:S}.
\end{prop}
\begin{proof}
	Let $c_i= p_i (q_{a}(i)-q_{s}(i)+u({i}))\Delta t$\\
	$
	d_i=\begin{cases}
	p_d(q_{a}(k)-q_{s}(k)+u_{k})& k \in T_{on}\\
	0 & \text{otherwise.}
	\end{cases}
	$
\end{proof}
We conclude that our approach to solving NFS DP problems can be applied to battery scheduling. That is, battery scheduling can be represented as an augmented DP problem of Form $\mcl A(t_0,x_0)$.

\subsection{Numerically Solving The Deterministic Battery Scheduling Problem} \label{sec:numerically solve deterministic battery problems}
Our proposed approximation scheme can be applied to solve the battery scheduling problem, $\mcl D(0,e_0)$. This is done by creating an augmented state variable based on the maximum function in the objective function, as in Section \ref{sec:ADP}, and thus constructing an equivalent DP problem of the form $\mcl A(0,x_0)$ \eqref{augmentation}; which is a special case of $\mcl P(t_0,x_0)$. Figure \ref{fig:Cost_vs_grid_points} shows how the monthly cost decreases when we use policies constructed from the associated discretized DP problems, $\mcl P_k(t_0,x_0)$, as $k$ is increased. Although we do not get a monotonically decreasing sequence of costs, the error does decrease as $k \to \infty$. Figure \ref{fig:peak_vs_grid_points} also shows that augmenting and then following our proposed discretization scheme for the battery scheduling problem results in a policy that reduces the consumption demand peak as $k$ is increased. Figure \ref{fig:Computational_time} shows {how the computational time required to solve the discretized battery scheduling problem appears to be of exponential nature with respect to the number of grid points.}

\begin{figure}
	\includegraphics[scale=0.6]{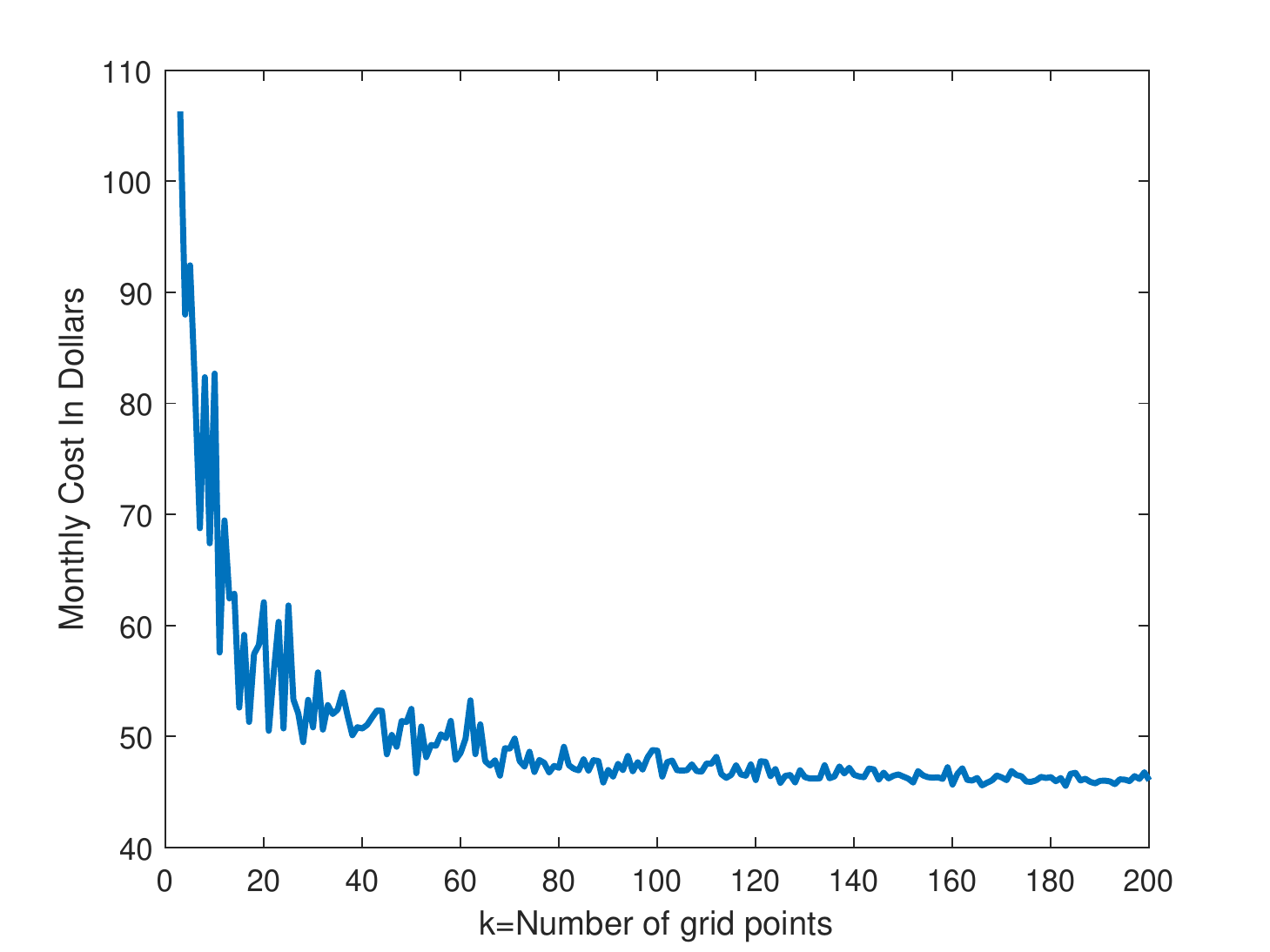}
	\vspace{-20pt}
	\caption{{The resulting monthly cost from using the policy found by solving the discretized problem, of form $P_k(t_0,x_0)$, for optimal battery scheduling.}}
	\label{fig:Cost_vs_grid_points}
	\vspace{-15pt}
\end{figure}

\begin{figure}
	\includegraphics[scale=0.6]{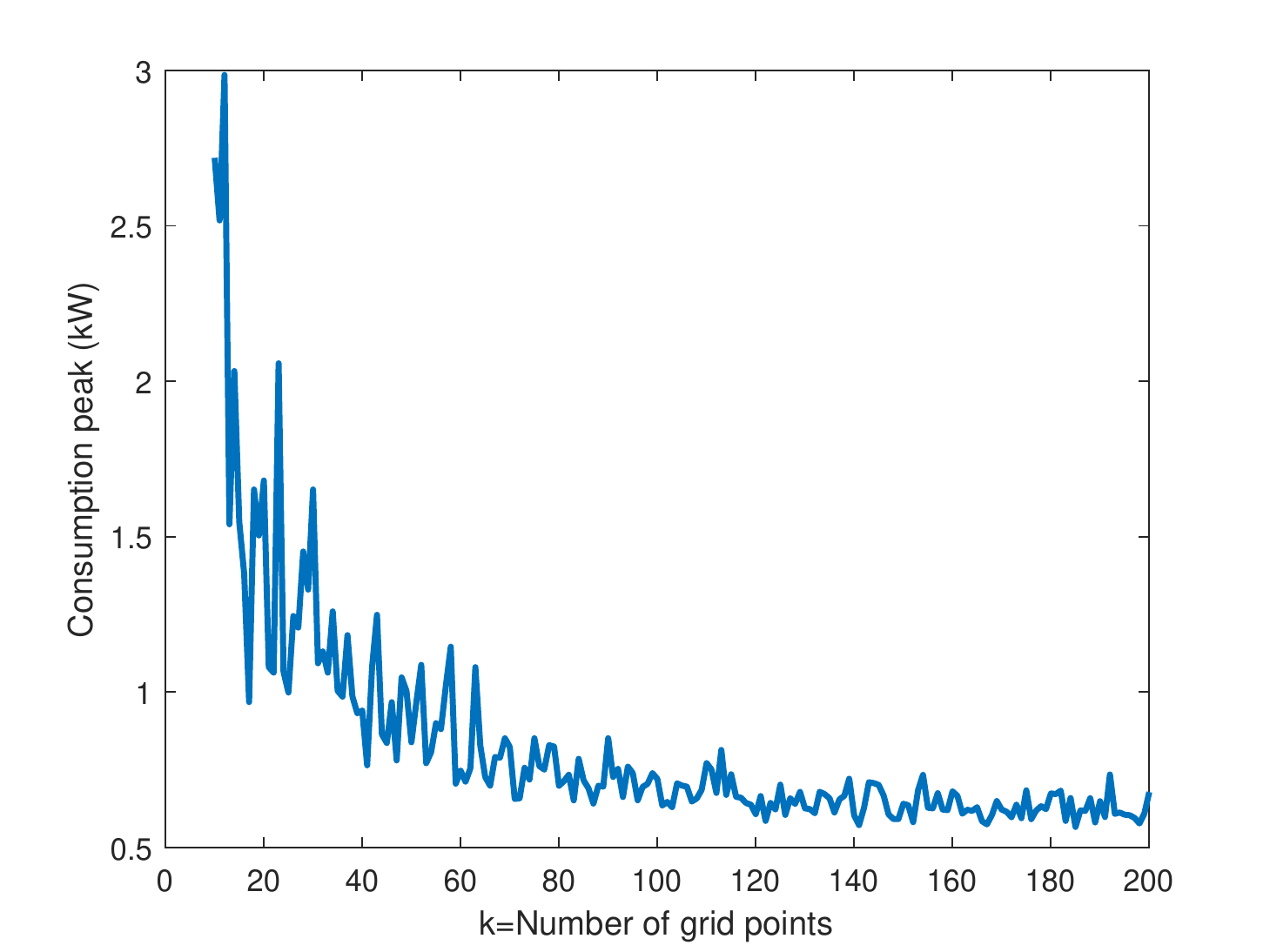}
	\vspace{-20pt}
	\caption{{The resulting maximum demand from using the policy found by solving the discretized problem, of form $P_k(t_0,x_0)$, for optimal battery scheduling.}}
	\label{fig:peak_vs_grid_points}
	\vspace{-15pt}
\end{figure}

\begin{figure}
	\includegraphics[scale=0.6]{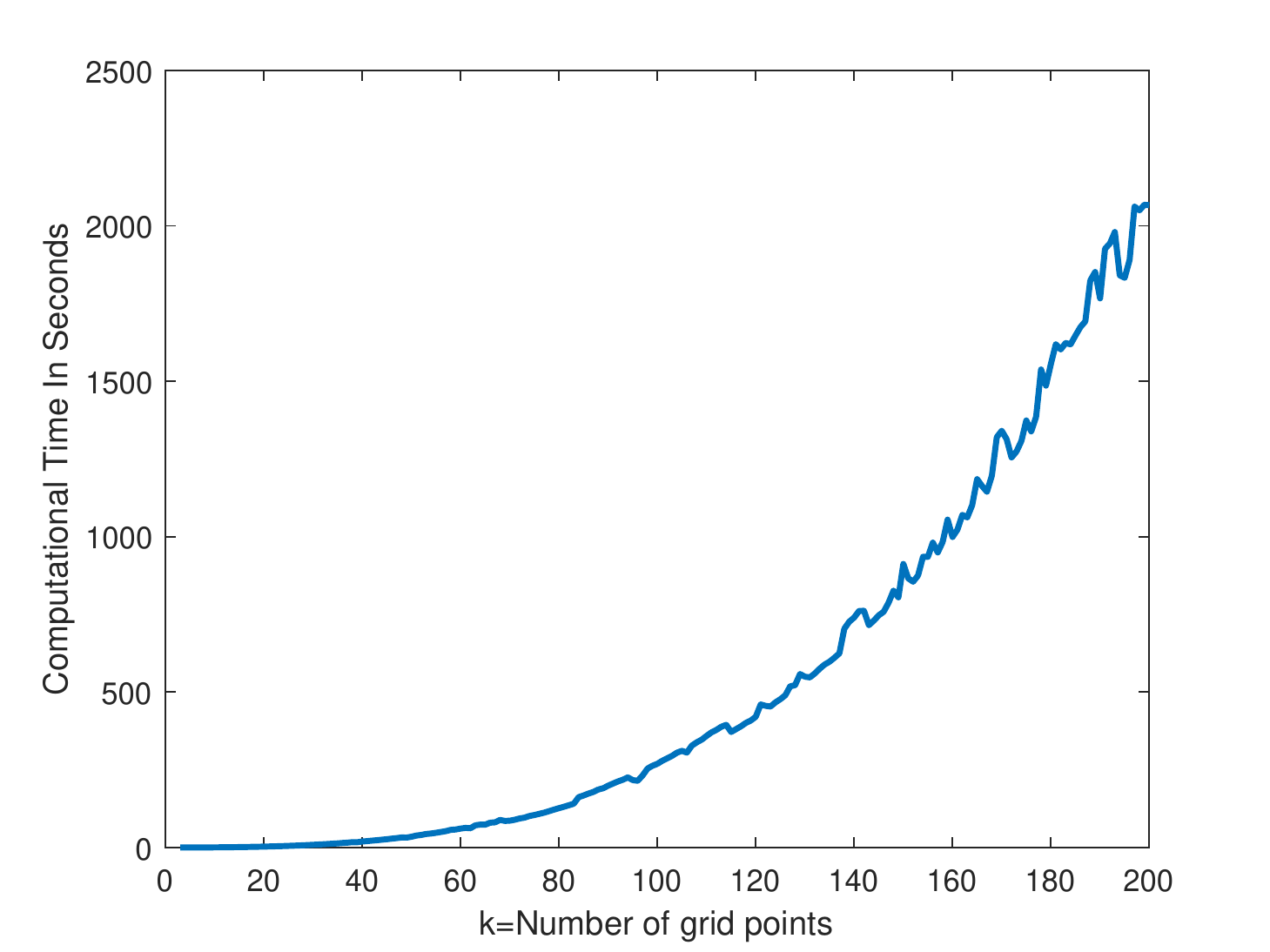}
	\vspace{-20pt}
	\caption{{The computational time in seconds required to solve the discretized battery scheduling problem, of form $P_k(t_0,x_0)$.}}
	\label{fig:Computational_time}
	\vspace{-15pt}
\end{figure}

\begin{figure*}
	\centering
	\includegraphics[width=18cm,height=8cm]{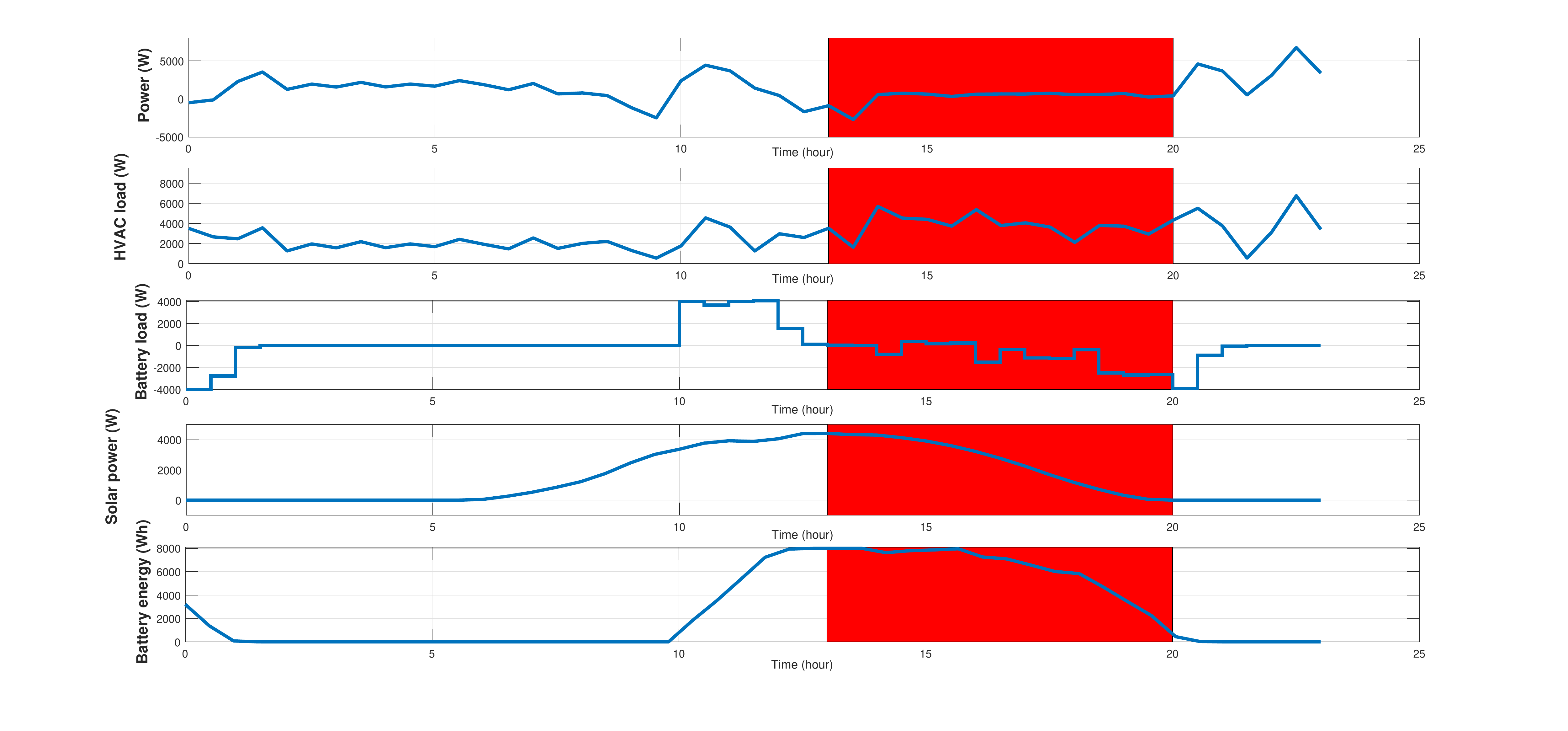}
	\vspace{-20pt}
	\caption{The trajectory the algorithm produces for deterministic solar data. The supremun of the power is 0.7033(kw) and the cost is \$46.389.}
	\label{fig:deterministic}
	\vspace{-15pt}
\end{figure*}

\small{
	\begin{table}
		\centering
		\caption{List of constant values (prices correspond to Salt River Project E21 price plan)}
		\label{tab:constants}
		\begin{tabular}{|l|l||l|l|l}
			\cline{1-4}
			Constant  & Value       &  Constant  & Value &  \\ \cline{1-4}
			$\alpha$  & 0.999791667 (W/h) &  $t_{\text{off}}$  & 41  &  \\ \cline{1-4}
			$\eta$    & 0.92  (\%)      &  $p_{\text{on}}$    & $0.0633 \times 10^{-3}$ (\$/KWh) &  \\ \cline{1-4}
			$\bar{u}$ & 4000 (Wh)       &  $p_{\text{off}}$ & $0.0423 \times 10^{-3}$ (\$/KWh) &  \\ \cline{1-4}
			$\underline{u}$  & -4000 (Wh) &  $p_{\text{d}}$  & 0.2973 (\$/KWh) &  \\ \cline{1-4}
			$\bar{e}$    & 8000   (Wh)     &  $\Delta t$    & 0.5 (h) &  \\ \cline{1-4}
			$t_{\text{on}}$ & 27        &  &  &  \\ \cline{1-4}
		\end{tabular}
	\end{table}
}
\normalsize

We used solar and usage data obtained by local utility Salt River Project (SRP) in Tempe, AZ, for power variables $q_s$ and $q_a$. We also use pricing data from SRP for the parameters $p_{\text{on}}$, $p_{\text{off}}$ and $p_d$. Battery data obtained for the Tesla Powerwall was used to determine the parameters $\alpha$, $\eta$, $\bar{u}$, $\underline{u}$ and  $\bar{e}$. The results of the simulation are shown in Figure \ref{fig:deterministic}. The policy used for this simulation was created using our augmentation and approximation scheme with $k=20$. Interpolation was used to aid in solving {Bellman's equation} \eqref{eqn:bellman} and decrease the approximation error. These results show an improvement in accuracy over results obtained for a similar problem in~\cite{multi-objective} (approximately \$0.98 savings). As expected, we see the battery charges during off-peak and then discharges during on peak times to reduce ToU charges, while maintaining a reserve which it uses to keep consumption flat during on peak times, thereby minimizing the demand charge. As a result the power stabilizes during on peak times - becoming constant.

\subsection{Solving The Stochastic Battery Scheduling Problem} \label{sec:Solving the stochastic battery scheduling problem}
To evaluate the effect of stochastic uncertainty on battery scheduling, we identify a Gauss-Markov model of solar generation based on SRP data. We construct the battery scheduling problem {of} the form $\mcl H_s(t_0,x_0)$ \eqref{eqn:Forward sep DP stochastic} and then use our proposed state augmentation approach to construct an equivalent DP problem of form $\mcl Q(t_0,x_0)$ \eqref{eqn:DP stochastic}. The problem of form $\mcl Q(t_0,x_0)$ is then solved approximately using the methodology of Section \ref{sec: approximately solving Q}.

\subsection{Solar Generation Model}
Our approach to modeling the dynamics of solar generation is based on \cite{Solar}. Our Markov type model can be used to generate high resolution data over large time horizons. The Markov property of the model results in deviation from the mean being correlated time to time, helping represent the physical phenomena of clouds gradually passing overhead rather than instantaneously appearing. 

Our model is a type of autoregressive-moving-average model (ARMAX) \cite{Wan_2015}. In \cite{Li_2014} it is seen ARMAX models preform better than auto-regressive integrated moving average (ARIMA) and in \cite{Chen_2015} it is shown ARMAX models can produce data similar to real data for local sites in Califonia and Colorado.

Exogenous variables, temperature and humidity, are included as state variables in addition to the primary variable - solar radiance. Cross correlations between state variables are computed from data. Specifically, we take time-series data of these quantities, denoted $\mbf {W}(t)$ and normalize this data as,
\begin{align*}
w_{i}(t)=\dfrac{W_{i}(t)-\mu_{i}(t)}{\sigma_{i}(t)}\vspace{-2mm},
\end{align*}
where $\mu_{i}(t)$ is the average historic and clear-sky mean of the variable $W_{i}$ at time step $t$ and $\sigma_{i}(t)$ is the standard deviation of variable $W_{i}$ at time step $t$.\\
The generating process is then given by:\vspace{-2mm}
\begin{align} \label{eqn: solar}
& \mbf{w}(t)=A \mbf {w}(t-1)+B \mbf {v}(t-1) \text{  for  } t=1,..,T\\ \nonumber
& \text{where } \mbf {w}(t) \in \mathbb{R}^3, \mbf {w}(0)= \mbf {0}\\ \nonumber
& \mbf v(t) \sim \mathcal{N}( \mbf {0},I_{3 \times 3}  ),
\end{align}
where the matrices $A$ and $B$ are chosen to preserve the lag 0 and lag 1 cross-correlations seen in the collected data. Specifically, we can compute these matrices as (\cite{Solar})\vspace{-2mm}
\begin{equation} \label{eqn: cross corellations}
 A=M_{1}M_{0}^{-1}\qquad \qquad \qquad  BB^{T}=M_{0}-M_{1}M_{0}^{-1}M_{1}^{T},
\end{equation}
where $M_{i}$ is the i-lag cross correlation matrix. So $(M_{i})_{m,n}=\rho_{i}(m,n)$ where $\rho_{i}(m,n)$ is the cross-correlation coefficient between variables m and n with variable n lagged by i time steps. Then, adding back in the mean and deviation, we obtain the power supplied by solar at time step $k$ as\vspace{-2mm}
\[
q_{s}(k)=w_{1}(k) \sigma_{1}(k) + \mu_{1}(k).\vspace{-2mm}
\]
Figure \ref{fig:solar} shows simulated irradiance data from our solar model when compared to actual recorded irradiance data. For this numerical implementation the mean and standard deviation, $(\mu_{i}(t))_{0 \le t \le T}$ and $(\sigma_{i}(t))_{0 \le t \le T}$, were calculated using data from Wunderground for a weather station in Tempe, AZ on October the 15th 2014 for each state variable. Cross correlations between the variables were also calculated from the same data set and \eqref{eqn: cross corellations} was solved giving the matrices $A$ and $B$ in \eqref{eqn: solar}. As seen in Figure~\ref{fig:solar} this solar generation model gives an output similar to what is observed in real data. Next we incorporate this model into our battery scheduling DP problems.

\begin{figure}
	\includegraphics[scale=0.63]{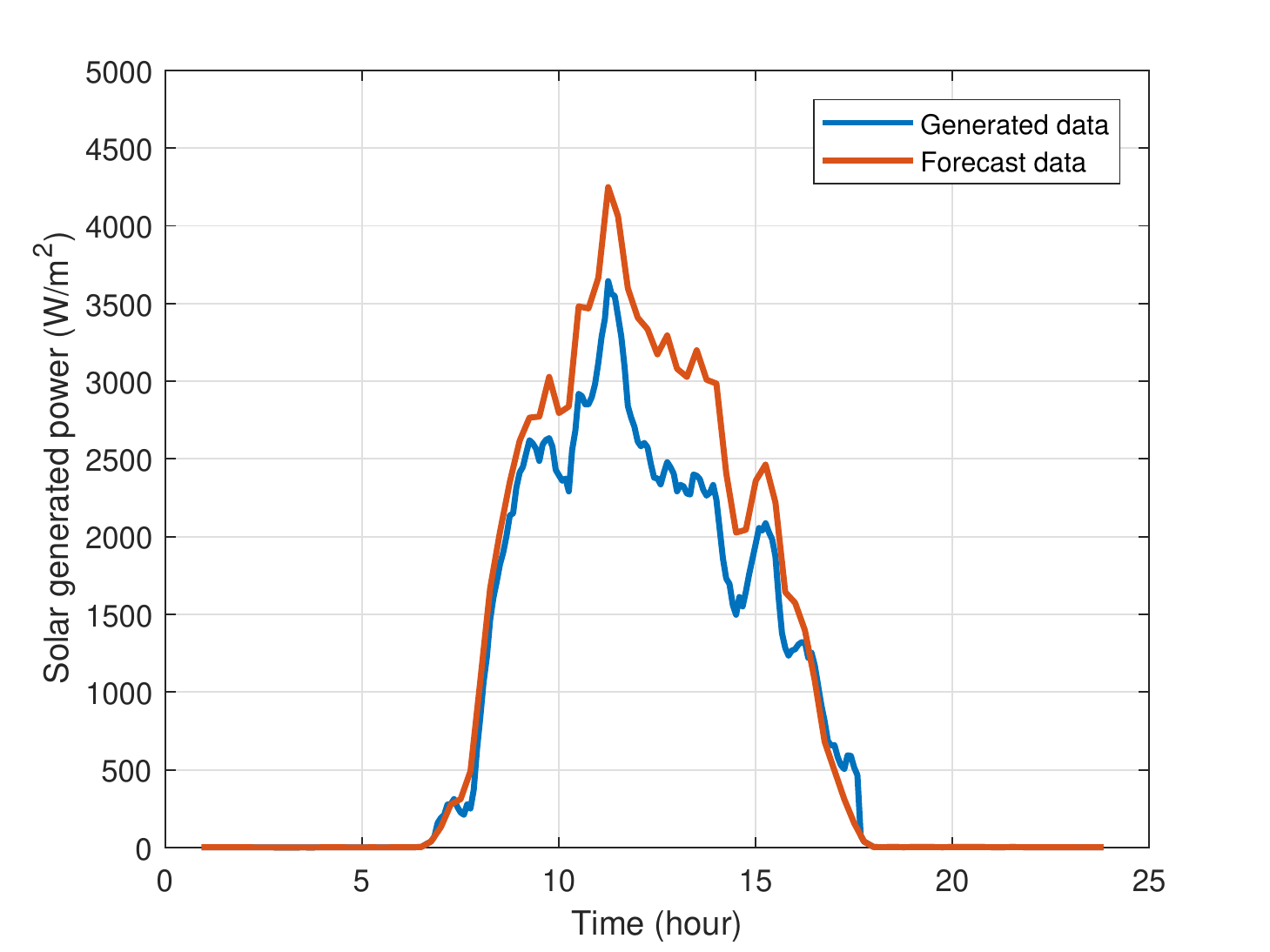}
	\vspace{-20pt}
	\caption{Solar data generated over 24 hours using data from Wunderground}
		\label{fig:solar}
		\vspace{-15pt}
\end{figure}

\noindent \textbf{Stochastic Battery Scheduling} We now modify Problem $\mcl D(0,e_0)$ \eqref{eqn: Determinist battery problem} to give a stochastic version of the battery scheduling problem $\mcl D_s(0,[e_0,0])$,

\small{ \begin{align} \label{eqn:Battery stochastic}
	&\arg \min_{\pi \in \Pi} \;\; \mbb E_{ [\mbf v]_{0}^{T-1} }\bigg[ J_{E}(\Phi_{f,0}(\pi, [e_0,0],T, [\mbf v]_{0}^{T-1})) \\ \nonumber
	& \qquad \qquad \qquad \qquad \qquad + J_{D}(\Phi_{f,0}(\pi, [e_0,0],T, [\mbf v]_{0}^{T-1})) \bigg] \\ \nonumber
	&\text{subject to:  }  \psi_{f,0}(\pi, [e_0,0], t, [\mbf v]_{0}^{t-1}) \in E_t \times \R^3 \text{ for  } t=0,..,T  \\ \nonumber
	&  \pi(x,t) \in U_t \text{ and }   v(t) \in \R^3 \sim \mathcal{N}(\mbf 0,I_{3 \times 3}) \forall x \in X_t, \forall t=0,..,T-1,
	\end{align} } \normalsize where $J_{E}$ is the ToU cost function and $J_{D}$ is the demand charge found in Section \ref{sec:battery dynamics}; $f([e,w],u,t,v)= \begin{bmatrix}
\alpha (e+\eta u\Delta t) \\ Aw+Bv  \end{bmatrix}$; $E_t=[\underline{e}, \bar{e}]$ and $U_t= [\underline{u}, \bar{u}]$ for all $t \in \{0,...,T\}$; $\psi_{f, t_0}$ and $\Phi_{f, t_0}$ are the state and trajectory map respectively from Definition \ref{defn: the trajectory map}; $[\mbf v]_{0}^{T-1} = [v(0),..,v(T-1)] \in \R^{3 \times T}$; matrices $A$ and $B$ are calculated from weather data using equations \eqref{eqn: cross corellations}; and all constants are found in Table \ref{tab:constants}.



\subsection{Numerically Solving The Stochastic Battery Scheduling Problem}
 Using the state augmentation procedure in Section \ref{sec:ADP} on the stochastic battery scheduling problem $\mcl D_s(0,[e_0,0])$ \eqref{eqn:Battery stochastic}, we may find a stochastic DP problem of the form $\mcl Q(t_0,x_0)$ \eqref{eqn:DP stochastic} such that an optimal policy for $\mcl D_s(0,[e_0,0])$ can be constructed from an optimal policy of $\mcl Q(t_0,x_0)$. We may then construct the approximated stochastic DP problem $\mcl Q_{\eps,k}(t_0,k)$ \eqref{eqn:DP stochastic_approx} and solve it using Bellman's equation \eqref{Bellman_stochastic}. From an optimal policy of $\mcl Q_{\eps,k}(t_0,k)$ we then construct a feasible policy for $\mcl Q(t_0,x_0)$ using \eqref{eqn:constructed policy_stochastic}. Figure \ref{fig:stochastic} demonstrates a simulation of using the feasible policy obtained via augmenting and approximating the stochastic battery scheduling problem with a reasonably selected family of compact state spaces, $\{H_{\eps,t}\}_{0 \le t \le T}$, and discretization level $k=10$. To {simplify} computation we use a one state version of our solar model \eqref{eqn: solar} and use interpolation while solving Bellman's equation. As expected the battery charges during the on peak times and conservatively discharges during the off-peak times. The solar data generated from this run are then used as input to the deterministic algorithm in order to compare performance. As anticipated, the deterministic case performs better than the stochastic case.

\section{Conclusion}
In this paper we propose a {general} formulation of the DP problem. We show that if the objective function is forward separable, DP problems may reformulated using state augmentation as an equivalent DP problem with additively separable objective function. Furthermore, we define a class of functions, called naturally forward separable (NFS) functions, such that DP problems with an objective function of this class can be tractably solved using state augmentation. Moreover, we show that the problem of optimal scheduling of battery storage in the presence of combined demand and time-of-use charges is a special case of this class of NFS DP problems. We further extend these results to stochastic DP problems with a NFS objective. The proposed algorithms are applied to a battery scheduling problem using first a deterministic and then Gauss-Markov model for solar generation and load.

{Extensions of this work include the  use of non-separable input constraints (such as those considered in \cite{mannor2016robust}) and algorithms for finding the minimal dimension NFS representation of a given objective function.}

\section{Appendix}

	\begin{lem} \label{lem: Optimal policy is Markov}
		{ Consider a DP problem of the form $\mcl Q(t_0,x_0)$ \eqref{eqn:DP stochastic} with additively separable objective function $J_{t_0}^{\mcl Q}$. For any family of functions of the form $\hat{\pi}: \R^{n \times (t-t_0+1)} \to \R^m$ are such $\hat{\pi}_t([(x(t_0),....,x(t))]) \in U_t $ and $f(x(t),\hat{\pi}_t([(x(t_0),....,x(t))]),t, v(t) ) \in X_{t+1}$ for all $x(i) \in X_i$, $i \in \{t_0,...,t\}$, $v(t) \in \R^q$ and $t \in \{t_0,...,T-1\}$ there exists $\alpha \in \Pi$ such that
		{ \begin{align} \label{identity: markov produces same objective function value}
			& \mbb E_{ [\mbf v]_{t_0}^{T-1}}\left( J_{t_0}^{\mcl Q}(\Phi_{f,t_0}(\alpha, x_0,T, [\mbf v]_{t_0}^{T-1}))  \right)\\ \nonumber
			& \qquad \qquad = \mbb E_{[\mbf v]_{t_0}^{T-1}}\left( J_{t_0}^{\mcl Q}(\Phi_{f,t_0}(\hat{\pi}, x_0,T, [\mbf v]_{t_0}^{T-1})) \right)
			\end{align} } \normalsize
		where we make a small abuse of notation to extend the trajectory map $\Phi_{f,t_0}$ to policies that use the entire state space history. }
	\end{lem}
	\begin{proof}
	{	Proposition 8.1 \cite{Bertsekas_1978} or  Theorem 6.2 \cite{Hinderer_1970}. }\end{proof}


\section*{Acknowledgments}
This research was supported by the NSF Grant CNS-1739990. We gratefully acknowledge data and expertise provided by Robert Hess of Salt River Project.

\bibliographystyle{ieeetr}
\bibliography{DP_submit}
\vspace{-1.2cm}
\begin{IEEEbiography}[{\includegraphics[width=1in,height=1.25in,clip,keepaspectratio]{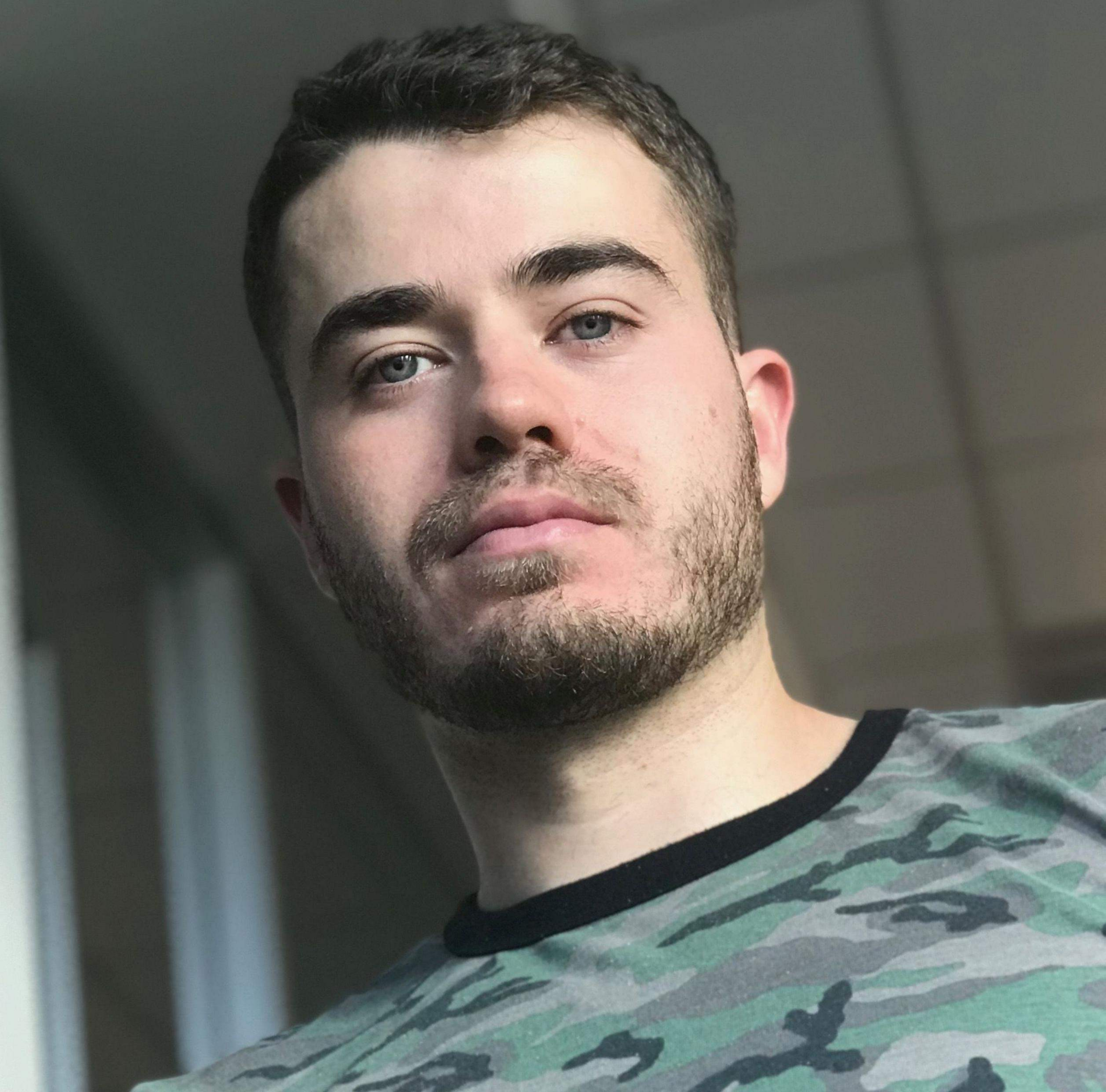}}]{Morgan Jones}
received the B.S. and Mmath in
mathematics from The University of Oxford, England in 2016.
He is a research associate with Cybernetic
Systems and Controls Laboratory (CSCL) in the School for Engineering
of Matter, Transport and Energy (SEMTE)
at Arizona State University. His research primarily
focuses on the estimation of attractors, reachable sets and regions of attraction for nonlinear ODE's. Furthermore,
He has been studying the effects of optimal energy storage in smart grid
environments.
\end{IEEEbiography}
\vspace{-1.2cm}
\begin{IEEEbiography}[{\includegraphics[width=1in,height=1.25in,clip,keepaspectratio]{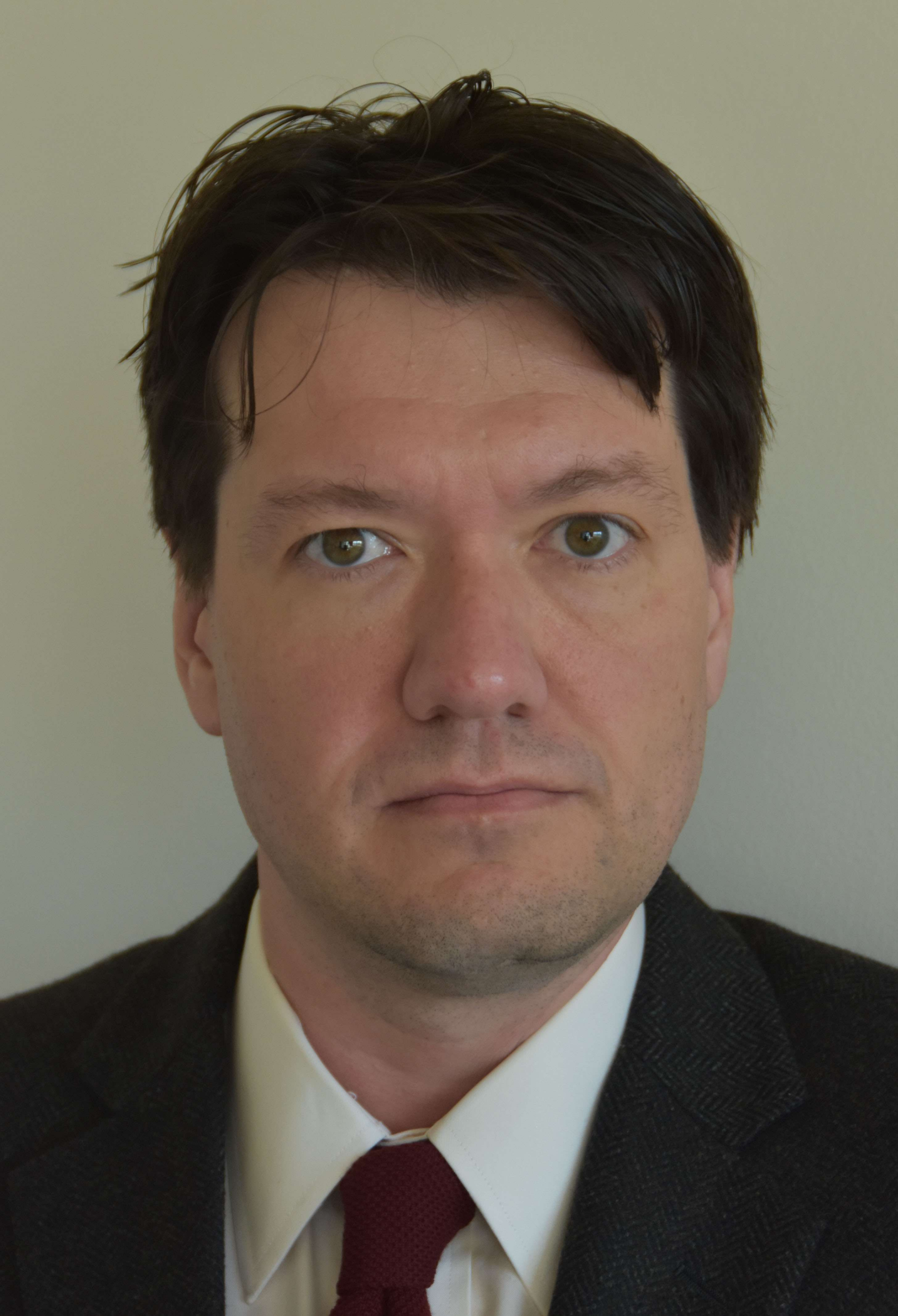}}]{Matthew M. Peet}
	received the B.S. degree in
	physics and in aerospace engineering from the University
	of Texas, Austin, TX, USA, in 1999 and
	the M.S. and Ph.D. degrees in aeronautics and astronautics
	from Stanford University, Stanford, CA,
	in 2001 and 2006, respectively. He was a Postdoctoral
	Fellow at INRIA, Paris, France from 2006 to
	2008. He was an Assistant Professor of Aerospace
	Engineering at the Illinois Institute of Technology,
	Chicago, IL, USA, from 2008 to 2012. Currently, he
	is an Associate Professor of Aerospace Engineering
	at Arizona State University, Tempe, AZ, USA. Dr. Peet received a National
	Science Foundation CAREER award in 2011.
\end{IEEEbiography}

\end{document}